\documentclass[11pt,reqno]{amsart}
\usepackage{amsmath,epsfig,graphicx,color}
\usepackage{ifthen}
\usepackage{amssymb,latexsym}
\usepackage{subfigure}
\usepackage{hyperref}
\hypersetup{
urlcolor=black, 
  menucolor=black, 
  citecolor=black, 
  anchorcolor=black, 
  filecolor=black, 
  linkcolor=black, 
  colorlinks=true,
}
\newcommand{\wt }{\widetilde}

\newcommand{\pro}{probabilit}

\newcommand{\bfw}{{\mathbf w}}
\newcommand{\bfv}{{\mathbf v}}
\newcommand{\bfA}{{\mathbf A}}
\newcommand{\bfT}{{\mathbf T}}
\newcommand{\bfR}{{\mathbf R}}
\newcommand{\bfD}{{\mathbf D}}
\newcommand{\bfN}{{\mathbf N}}
\newcommand{\bfM}{{\mathbf M}}
\newcommand{\bfF}{{\mathbf F}}
\newcommand{\bfH}{{\mathbf H}}
\newcommand{\bfP}{{\mathbf P}}

\newcommand{\fct}{function}
\newcommand{\slvary}{slowly varying}
\newcommand{\regvar}{regular variation}
\newcommand{\regvary}{regularly varying}
\newcommand{\st}{such that}
\newcommand{\stas}{\stackrel{\rm a.s.}{\rightarrow}}
\newcommand{\std}{\stackrel{\rm d}{\rightarrow}}
\newcommand{\stp}{\stackrel{\P}{\rightarrow}}
\newcommand{\la}{\lambda}
\newcommand{\ds}{distribution}
\newcommand{\beao}{\begin{eqnarray*}}
\newcommand{\eeao}{\end{eqnarray*}}
\newcommand{\beam}{\begin{eqnarray}}
\newcommand{\eeam}{\end{eqnarray}}

\definecolor{darkblue}{rgb}{.1, 0.1,.8}
\definecolor{darkgreen}{rgb}{0,0.8,0.2}
\definecolor{darkred}{rgb}{.8, .1,.1}
\newcommand{\bco}{\begin{corrolary}}
\newcommand{\eco}{\end{corrolary}}

\newcommand{\E}{\mathbb{E}}
\renewcommand{\P}{\mathbb{P}}
\newcommand{\1}{\mathbf{1}}
\newcommand{\R}{\mathbb{R}}

\newcommand{\bfC}{{\mathbf C}}
\newcommand{\bfB}{{\mathbf B}}
\newcommand{\Z}{\mathbb{Z}}

\newcommand{\Frechet}{Fr\'{e}chet }

\newcommand{\Y}{{\mathbf Y}}
\newcommand{\M}{{\mathbf M}}
\newcommand{\A}{{\mathbf A}}
\newcommand{\bfZ}{{\mathbf Z}}
\newcommand{\z}{{\mathbf Z}}

\newcommand{\dint}{\,\mathrm{d}}

\newcommand{\twonorm}[1]{\|#1\|_2}
\newcommand{\inftynorm}[1]{\|#1\|_\infty}
\newcommand{\frobnorm}[1]{\|#1\|_F}
\newcommand{\ltwonorm}[1]{\|#1\|_{\ell_2}}
\newcommand{\vep}{\varepsilon}
\newcommand{\nto}{n \to \infty}
\newcommand{\xto}{x \to \infty}
\newcommand{\lhs}{left-hand side}
\newcommand{\rhs}{right-hand side}
\newcommand{\ts}{time series}

\newcommand{\rv}{random variable}
\newcommand{\tr}{\operatorname{tr}}

\newcommand{\diag}{\operatorname{diag}}
\newcommand{\Dr}{D^{\rightarrow}}
\newcommand{\Dc}{D^{\downarrow}}
\newcommand{\bfe}{{\mathbf e}}
\newcommand{\slln}{strong law of large numbers}
\newcommand{\clt}{central limit theorem}

\newcommand{\ex}{{\rm e}\,}

\def\tag{\refstepcounter{equation}\leqno }

\newtheorem{lemma}{Lemma}[section]

\newtheorem{theorem}[lemma]{Theorem}
\newtheorem{proposition}[lemma]{Proposition}

\newtheorem{corollary}[lemma]{Corollary}

\newtheorem{remark}[lemma]{Remark}

\newcommand{\cid}{\stackrel{d}{\rightarrow}}
\newcommand{\cip}{\stackrel{\P}{\rightarrow}}
\newcommand{\civ}{\stackrel{v}{\rightarrow}}

\newcommand{\pp}{point process}
\newcommand{\con}{convergence}
\newcommand{\seq}{sequence}
\newcommand{\ms}{measure}
\newcommand{\asy}{asymptotic}
\begin{document}
\today
\title[Eigenvalues and eigenvectors of heavy-tailed sample covariance matrices]{Eigenvalues and eigenvectors of heavy-tailed sample covariance matrices with general growth rates: the iid case}
\thanks{Johannes Heiny's and Thomas Mikosch's research is partly supported by the Danish Research Council Grant DFF-4002-00435 ``Large random matrices with heavy tails and dependence''.}
\author[Johannes Heiny]{Johannes Heiny}
\author[Thomas Mikosch]{Thomas Mikosch}
%\author[Xiaolei Xie]{Xiaolei Xie}
\address{Department  of Mathematical Sciences,
University of Copenhagen,
Universitetsparken 5,
DK-2100 Copenhagen,
Denmark}
\email{johannes.heiny@math.ku.dk}
\email{mikosch@math.ku.dk\,, www.math.ku.dk/$\sim$mikosch}
%\email{xie.xiaolei@gmail.com}

\begin{abstract}
In this paper we study the joint \ds al \con\ of the largest eigenvalues of the sample covariance matrix
of a $p$-dimensional \ts\ with iid entries when $p$ converges to infinity together with the sample size $n$. We consider
only heavy-tailed \ts\ in the sense that the entries satisfy some \regvar\ condition which ensures that
their fourth moment is infinite. In this case, Soshnikov~\cite{soshnikov:2004,soshnikov:2006} and 
Auffinger et al.~\cite{auffinger:arous:peche:2009} proved the weak \con\ of the \pp es of the normalized eigenvalues 
of the sample covariance matrix towards an inhomogeneous Poisson process which implies in turn that the largest eigenvalue
converges in \ds\ to a Fr\'echet distributed \rv . They proved these results under the assumption that $p$ and $n$ are proportional
to each other. In this paper we show that the aforementioned results remain valid if $p$ grows at any polynomial rate. The proofs 
are different from those in \cite{auffinger:arous:peche:2009,soshnikov:2004,soshnikov:2006}; we employ large deviation 
techniques to achieve them. The proofs reveal that only the diagonal of the sample covariance matrix is relevant for the \asy\ behavior of the largest eigenvalues and the corresponding eigenvectors which are close to the canonical basis vectors.  
We also discuss extensions of the results to sample autocovariance matrices.
\end{abstract}

\maketitle
%\tableofcontents

%-------------------------------------------------------------------
\section{Introduction}\label{sec:motivation}\setcounter{equation}{0}
In recent years we have seen a vast increase in the number and sizes of data sets. 
Science  (meteorology, telecommunications, genomics, \ldots), society (social networks, finance, military and civil intelligence, \ldots) and industry 
need to extract valuable information from high-dimensional data sets
which are often too large or 
complex to be processed by traditional means. 
In order to explore the structure of data one often studies the dependence via 
(sample) covariances and correlations. %Here we focus on the {\em sample covariance matrix}.
Often dimension reduction techniques facilitate further analyzes of large data matrices. For example, 
{\em principal component analysis}  (PCA) transforms the data linearly such that only a few of the resulting vectors 
contain most of the variation in the data. These {\em principal component vectors} 
are the eigenvectors associated with the largest eigenvalues of the sample covariance matrix.
\par 
The aim of this paper is to investigate the \asy\ properties of the 
largest {\em eigenvalues} and their corresponding {\em eigenvectors} for
sample covariance matrices of high-dimensional heavy-tailed time series with iid entries.
Special emphasis is given to the case 
%we study the largest {\em eigenvalues} and the corresponding {\em eigenvectors} of the {\em sample covariance matrix} 
when the dimension $p$ and the sample size $n$ tend to infinity simultaneously, not necessarily at the same rate. 
%Classical limit theorems in probability theory and statistics are based on the assumption of finite dimension of data. 
%However, over the last decade data sets gathered by cheap and widely available devices have grown in 
%size exponentially. The question whether classical or large-dimensional limit results are closer to reality is addressed by Bai and Silverstein \cite[Chapter~1]{bai:silverstein:2010}.
\par
Throughout we consider the $p\times n$ data matrix 
\beao
\bfZ =\bfZ_n = \big(Z_{it}\big)_{i=1,\ldots,p;t=1,\ldots,n}\,
\eeao
A column of $\bfZ$ represents an observation of a $p$-dimensional time series.
We assume that the entries $Z_{it}$ are real-valued, independent and identically distributed (iid), 
unless stated otherwise. We write $Z$ for a generic element and assume $\E[Z]=0$ and $\E [Z^2]=1$ 
if the first and second moments of $Z$ are finite, respectively.
We are interested in limit theory for the eigenvalues $\lambda_1,\ldots,\lambda_p$ of the {\em sample covariance matrix} $\bfZ\bfZ'$
and their ordered values 
\beam\label{eq:order}
\la_{(1)}\ge \cdots \ge \la_{(p)}\,.
\eeam
In this notation we suppress 
the dependence of $(\la_i)$ on $n$. We will only discuss the case when $p \to \infty$; for the finite $p$ case 
we refer to \cite{anderson:1963,muirhead}.

\subsection{The light-tailed case}
In random matrix theory a lot of attention has been given to the {\em empirical spectral distribution \fct } of 
the sequence $(n^{-1}\z \z')$:
\begin{equation*}
F_{n^{-1}\z \z'}(x)= \frac{1}{p}\; \# \{ 1\le j\le p : n^{-1} \lambda_j \le x \}, \qquad x\ge 0\,,\quad n\ge1.
\end{equation*}
In the literature convergence results for $(F_{n^{-1}\z \z'})$  
are established under the assumption that $p$ and $n$ grow at the same rate:
\beam\label{eq:gamma}
\dfrac p n\to \gamma\qquad \mbox{for some $\gamma \in (0,\infty)$.}
\eeam 
Suppose that the iid entries $Z_{it}$ have mean $0$ and variance $1$. If \eqref{eq:gamma} holds then, with probability one, 
$(F_{n^{-1}\z \z'})$ converges to the  Mar\v cenko--Pastur law with absolutely  continuous part given by the density,
\begin{eqnarray}\label{eq:MP}
f_\gamma(x) = 
\left\{\begin{array}{cc}
\frac{1}{2\pi x\gamma} \sqrt{(b-x)(x-a)} \,, & \mbox{if } a\le x \le b, \\
 0 \,, & \mbox{otherwise,}
\end{array}\right.
\end{eqnarray}\noindent
where $a=(1 -\sqrt{\gamma})^2$ and $b=(1 +\sqrt{\gamma})^2$. For  $\gamma>1$
the Mar\v cenko--Pastur law has an additional point mass $1-1/\gamma$ at the origin; see 
Bai and Silverstein~\cite[Chapter~3]{bai:silverstein:2010}. 
This  mass is intuitively explained by the fact that, with probability $1$, $\min(p,n)$ 
eigenvalues $\lambda_i$ are non-zero. When $n=(1/\gamma) \; p$ and $\gamma >1$ the fraction of non-zero eigenvalues 
is $1/\gamma$ while the fraction of zero eigenvalues is $1-1/\gamma$.
\par
The moment condition $\E[Z^2]<\infty$ is crucial for deriving the Mar\v cenko--Pastur limit law. 
When studying the largest eigenvalues of the sample covariance matrix $\z\z'$ the moment condition $\E[Z^4]<\infty$ plays a similarly 
important role; we assume it in the remainder of this subsection. 
%It will be convenient to write
%\beam\label{eq:order}
%\la_{(1)}\ge \cdots \ge \la_{(p)}
%\eeam
%for the ordered eigenvalues of $\z\z'$ where we suppress the dependence on $n$ in the notation. 
If \eqref{eq:gamma} holds and the iid entries $Z_{it}$ have zero mean and unit variance, %and finite fourth moment,  
Geman~\cite{geman} showed that
\beam\label{eq:geman}
\dfrac {\la_{(1)}}{n} \stas \big(1+\sqrt{\gamma}\big)^2\,,\qquad \nto\,.
\eeam
This means that $\la_{(1)}/n$ converges to the right endpoint of the Mar\v cenko--Pastur law in \eqref{eq:MP}.
Johnstone \cite{johnstone:2001} complemented this \slln\ by the corresponding \clt\ in the special case of iid standard normal 
entries:
\beam\label{eq:tc}
\frac{\lambda_{(1)}-\mu_{n,p}}{\sigma_{n,p}} \cid \xi,
\eeam
where the limiting \rv\ has a {\em Tracy--Widom \ds} of order 1 and the centering and scaling constants are
\begin{equation*}
\mu_{n,p}=(\sqrt{n-1}+ \sqrt{p})^2, \quad \sigma_{n,p}=(\sqrt{n-1}+ \sqrt{p}) \Big( \frac{1}{\sqrt{n-1}} +\frac{1}{\sqrt{p}}\Big)^{1/3}\,;
\end{equation*}
%This \ds\ appears as a generic limit \ds\ in random matrix theory. 
%Its distribution function $F_1$ is given by
%\begin{equation*}
%F_1(s) = \exp\Big\{
%  -\frac{1}{2} \int_{s}^\infty [
%    q(x) + (x - s) q^2(x)
% ] \dint x
%\Big\}
%\end{equation*}
%where $q(x)$ is the unique solution to the Painlev\'e II differential
%equation
%\begin{equation*}
%  q''(x) = xq(x) + 2 q^3(x)\,,
%\end{equation*}
%where $ q(x)\sim {\rm Ai}(x)$ as $x \to \infty$ and
%Ai$(\cdot)$ is the Airy kernel; 
see Tracy and Widom~\cite{tracy:widom:2012} for details. 
Ma~\cite{zongming:2012} showed Berry--Esseen-type bounds for  
\eqref{eq:tc}. %He show improves if one slightly modifies the centering and scaling of $\lambda_{(1)}$. 
%He proved that ``the difference between the distribution of $(\lambda_{(1)}-\widetilde{\mu}_{n,p})/\widetilde{\sigma}_{n,p}$ and $F_1$ reduces to `second order', being $O((n \wedge p)^{-2/3})$ rather than $O((n \wedge p)^{-1/3})$, that would apply by using \eqref{eq:tc}'', where the modified centering and scaling constants are
%\begin{equation*}
%\widetilde{\mu}_{n,p}=(\sqrt{n-1/2}+ \sqrt{p-1/2})^2, \quad \widetilde{\sigma}_{n,p}=(\sqrt{n-1/2}+ \sqrt{p-1/2}) \Big( \frac{1}{\sqrt{n-1/2}} +\frac{1}{\sqrt{p-1/2}}\Big)^{1/3}.
%\end{equation*}
\par
%{ The calculation of the spectrum} % is facilitated by the fact that the distribution of 
%the classical Gaussian matrix ensembles is invariant under orthogonal transformations. The corresponding  
%computation for non-invariant 
Asymptotic theory for the largest eigenvalues of sample covariance matrices with 
non-Gaussian entries is more complicated; %and was a major challenge for several years; 
pioneering work is due to Johansson \cite{johansson}. 
Johnstone's result was extended to matrices $\z$ with iid non-Gaussian entries 
by Tao and Vu \cite[Theorem~1.16]{tao09b}, assuming that the first four moments of $Z$ match  those of 
the normal \ds. % they showed \eqref{eq:tc} by
%employing {\em Lindeberg's replacement method}, i.e., the iid non-Gaussian entries are replaced 
%step-by-step by iid Gaussian ones.
%This approach is well-known from summation theory for \seq s of iid \rv s. 
Tao and Vu's result is a consequence of the so-called {\em Four Moment Theorem} 
which describes the insensitivity of the eigenvalues with respect to changes in the distribution of the entries. 
To some extent (modulo the strong moment matching conditions) it shows the universality of Johnstone's limit 
result \eqref{eq:tc}.% see Figure~\ref{fig:3point} for an illustration.
%\begin{figure}[htb!]
%  \centering
%    \includegraphics[scale=0.550]{3point-TW.pdf}
%  \caption{Entry distribution: $\P(Z=\sqrt{3}) = \P(Z=-\sqrt{3} ) =
%  1/6$, $\P(Z=0)=2/3$. Note $\E [Z] = 0$, { $\E[ Z^2] = 1$, $\E [Z^3] = 0$ and
%  $\E[ Z^4] = 3$, i.e., the first 4 moments of $Z$ match those of the standard normal \ds .}}
%  \label{fig:3point}
%\end{figure}
\par
In the light-tailed case little is known when $p$ and $n$ grow at different rates, i.e., $\lim p/n \in\{ 0,\infty\}$. 
Notable exceptions are El Karoui \cite{elkaroui:2003} who proved that Johnstone's result (assuming iid standard normal entries)
remains valid when $p/n\to 0$ or $n/p\to\infty$, and P{\'e}ch{\'e} \cite{peche:2009} who showed universality results for the largest eigenvalues of some sample covariance matrices with non-Gaussian entries.
%While the iid assumption may be somewhat weakened (see \cite{bai:silverstein:2010}) and one still obtains the same convergence results of the empirical spectral distribution.
%There seems to be no corresponding result if \eqref{eq:gamma} is violated.

\subsection{The heavy-tailed case}

Distributions of which certain moments cease to exist are often called heavy-tailed. 
So far we reviewed theoretical results where the data matrix $\z$ was ``light-tailed" in the following sense:
for the \ds al convergence of the empirical 
spectral distribution  and the largest eigenvalue of the sample covariance matrix  towards the 
Mar\v cenko--Pastur 
and Tracy-Widom \ds s, respectively, we  
required finite second/fourth moments of the entries.
\par
The behavior of the largest eigenvalue $\la_{(1)}$ 
changes dramatically when $\E[Z^4]=\infty$.
Bai and Silverstein~\cite{baisilv} proved for an $n\times n$ matrix $\z$ with iid centered entries 
that
\beam\label{eq:wdfr}
\limsup_{\nto} \dfrac{\la_{(1)}}{n}=\infty \qquad {\rm a.s.}
\eeam
This is in stark contrast to Geman's result \eqref{eq:geman}. 
%It is also unlikely that the  
%Tracy--Widom \ds\ yields a good approximation to the \ds\ of the largest eigenvalue of the sample covariance matrix.
\par
Following classical limit theory for partial sum processes and maxima, we require more than an infinite fourth moment.
We assume a {\em \regvar\ condition} on the tail of $Z$:
\beam\label{eq:regvar}
\P(Z>x)\sim p_+\,\dfrac{L(x)}{x^{\alpha}}\qquad\mbox{and}\qquad \P(Z<-x) \sim p_-\,\dfrac{L(x)}{x^\alpha}\,,\qquad \xto\,,
\eeam
for some $\alpha\in (0,4)$, where $ p_\pm$ are non-negative constants \st\ $p_++p_-=1$ and $L$ is a \slvary\ \fct .  
We will also refer to $Z$ as a \regvary\ \rv , $\z$ as a \regvary\ matrix, etc.
%In particular, we have $\E [|X|^\delta]= \infty$ for any $\delta>\alpha$. Moments higher than $\alpha$ cease to exist. 
%Condition \eqref{eq:regvar} on $Z$  is needed for proving \asy\ theory for the eigenvalues of $\z\z'$.
%\par
%We are interested in the case $\alpha \in (0,4)$. %In view of \eqref{eq:wdfr} one has to choose a stronger normalization than $n$, at least when $p=n$. It turns out that $a_{np}^2$ with is the appropriate choice.
{\em Here and in what follows, we normalize the eigenvalues $(\la_i)$ by $(a_{np}^2)$ where the \seq\
$(a_k)$ is chosen \st } 
\beao
\P(|Z|>a_k)\sim k^{-1}\,,\quad k\to\infty.
\eeao
Standard theory for \regvary\ \fct s (e.g. Bingham et al. \cite{bingham:goldie:teugels:1987}, Feller \cite{feller})
%From standard references on extreme value theory (for example Embrechts et al.~\cite{embrechts:kluppelberg:mikosch:1997}) one knows 
yields that $a_n=n^{1/\alpha} \ell(n)$ where $\ell$ is a slowly varying function. Assuming \eqref{eq:gamma} for $p$,
the Potter bounds (see \cite[p.~25]{bingham:goldie:teugels:1987}) yield for $\alpha\in (0,4)$ that
\begin{equation}\label{eq:toinf}
 \frac{a_{np}^2}{n} \sim \frac{n^{4/\alpha} \gamma^{2/\alpha}\,\ell^2(n^2 \gamma)}{n} \to \infty, \qquad \nto\,,
\end{equation}
i.e., the normalization $a_{np}^2$ is stronger than $n$.
\par
The eigenvalues $(\lambda_i)$ of a heavy-tailed matrix $\z\z'$ were 
studied first by Soshnikov~\cite{soshnikov:2004,soshnikov:2006}. He showed under 
\eqref{eq:gamma} and \eqref{eq:regvar} for $\alpha \in (0,2)$ that  
\begin{equation}\label{eq:larg}
\dfrac{\lambda_{(1)}}{a_{np}^2} \cid \zeta, \quad \nto,
\end{equation}
where $\zeta$ follows a {\em \Frechet distribution} with parameter $\alpha/2$:
\beao
\Phi_{\alpha/2}(x) =\ex^{-x^{-\alpha/2}}\,,\qquad x>0\,.
\eeao
Later Auffinger et al.~\cite{auffinger:arous:peche:2009} established \eqref{eq:larg} also for $\alpha \in [2,4)$ 
under the additional assumption that the entries are centered. 
Both Soshnikov \cite{soshnikov:2004,soshnikov:2006}
and Auffinger et al.~\cite{auffinger:arous:peche:2009} proved convergence of the point processes of normalized eigenvalues, 
from which one can easily infer the joint limiting distribution of the $k$ largest eigenvalues.  
Davis et al.~\cite{davis:mikosch:pfaffel:2015,davis:pfaffel:stelzer:2014} 
extended these results allowing for more general growth of $p$ 
than dictated by \eqref{eq:gamma} and a linear dependence structure between the rows and columns of $\z$; see also Chakrabarty et al.~\cite{chakrabarty:hazra:roy:20013} and the overview
paper Davis et al.~\cite{davis:mikosch:heiny:xie:2015}. {\em The study of eigenvectors of heavy-tailed sample covariance matrices is a fresh topic, which has not been explored in the literature listed here.}
\par
For the sake of completeness we mention 
%In \eqref{eq:MP} we have seen that the sequence 
that, under \eqref{eq:gamma} with $\gamma\in (0,1]$,  
\eqref{eq:regvar} with $\alpha\in (0,2)$ and $\E[Z]=0$ if the latter expectation is defined,
the empirical spectral \ds\ $F_{a_{n+p}^{-2}\z \z'}$ converges weakly with \pro y one to a deterministic 
\pro y \ms\ whose density $\rho_{\alpha}^\gamma$ satisfies
\begin{equation*}
\rho_{\alpha}^\gamma(x) x^{1+\alpha/2} \to \frac{\alpha \gamma}{2(1+\gamma)}, \qquad x \to \infty\,,
\end{equation*}
%$(F_{n^{-1}\z\z'})$ of empirical spectral distribution converges to the Marchenko--Pastur if the centered iid entries possess a finite second moment. Now we will discuss the situation when the entries are still iid and centered, but have an infinite variane. For this aim we assume the entries to be regularly varying with index $\alpha \in (0,2)$. Then we obtain analogously to \eqref{eq:toinf} that $a_{n+p}^2/n \to \infty$. 
%Assuming \eqref{eq:gamma} with $\gamma\in (0,1]$ in this infinite variance case 
see Belinschi et al.~\cite[Theorem~1.10]{belinschi:dembo:guionnet:2009} 
%showed that the sequence $(F_{a_{n+p}^{-2}\z \z'})$ converges with probability one 
%to a non-random probability measure $\mu_{\alpha}^\gamma$, which has heavy tails. The corresponding density $\rho_{\alpha}^\gamma$ satisfies
and Ben Arous and Guionnet \cite[Theorem~1.6]{arous:guionnet:2008}.

%It is very interesting to note that the behavior of the eigenvalues itself changes dramatically if one moves from finite to infinite fourth moment, while the type of the limit of the empirical spectral distribution depends on the finiteness of the second moment. In the case of finite fourth moment, \eqref{eq:geman} and \eqref{eq:MP} imply that $\lambda_{(k)}/n$ has to converge to the right endpoint of the Marchenko--Pastur law for any finite $k$. However, if we assume infinite fourth moment but finite second moment the empirical spectral distributions  $(F_{{n}^{-1}\z \z'})$ still converge to the Marchenko--Pastur law with the same support, while $\lambda_{(1)}/n$ will tend to infintity due to \eqref{eq:wdfr}. This means that the proportion of the normalized eigenvalues taking values larger than $(1+\sqrt{\gamma})^2$ has to be negligible; otherwise we would have a contradiction to the convergence of the empirical spectral distributions to the Marchenko--Pastur law.
%Finally, if the second moment is infinity, i.e.~regularly varying entries with index $\alpha \in (0,2)$, then the much larger values of the eigenvalues are visible in the limiting spectral distribution even when using a much stronger normalization \cite[Theorem~1.10]{belinschi:dembo:guionnet:2009}. 

\subsection{Structure of the paper}
The primary objective of this paper is to study the joint distribution of the largest eigenvalues of the sample covariance matrix $\z\z'$
in the case of iid regularly varying entries with infinite fourth moment.
We make a connection between extreme value theory, point process \con\ and the behavior of the largest eigenvalues. 
We study these eigenvalues under polynomial 
growth rates of the dimension $p$ relative to the sample size $n$. 
It turns out that they are essentially determined by the extreme diagonal elements of $\bfZ\bfZ'$ or, alternatively, 
by the extreme order statistics of the squared entries of $\bfZ$.  
\par
In Section~\ref{sec:2} we consider power-law growth rates of $(p_n)$, 
thereby generalizing proportional growth as prescribed by \eqref{eq:gamma}.
Our main results are presented in Section~\ref{sec:mainresults}. 
Theorem~\ref{thm:iidmain} provides approximations of the ordered eigenvalues of the sample covariance matrix either by
the ordered diagonal elements of $\bfZ\bfZ'$ or $\bfZ'\bfZ$, or by the order statistics of the squared entries of $\bfZ$.
These approximations provide a clear picture where the largest eigenvalues of the sample covariance matrix originate from. 
Our results generalize those in Soshnikov~\cite{soshnikov:2004,soshnikov:2006} and 
Auffinger et al.~\cite{auffinger:arous:peche:2009} who assume proportionality of $p$ and $n$.
The employed techniques originate from extreme value analysis and large deviation theory; the proofs differ from those 
in the aforementioned literature. 
The same techniques can be applied when the entries of $\bfZ$ 
are heavy-tailed and allow for dependence through the rows and across the columns; see 
Davis et al.~\cite{davis:mikosch:pfaffel:2015,davis:pfaffel:stelzer:2014} for some recent attempts 
when the entries satisfy some linear dependence conditions. In the iid case, these results are covered by the present paper
and we also show that they remain valid under much more general growth conditions than in \cite{davis:mikosch:pfaffel:2015,davis:pfaffel:stelzer:2014}. In particular, we make clear that centering of the sample covariance matrix (as assumed in 
\cite{davis:mikosch:pfaffel:2015,davis:pfaffel:stelzer:2014} when $Z$ has a finite second moment) is not needed. Thus, our techniques are applicable under rather general dependence structures. We refer to the recent work by Janssen et al.~\cite{janssen:mikosch:rezapour:xie:2016} on eigenvalues of stochastic volatility matrix models, where non-linear dependence was allowed.
% provided the rate of $p$ is sufficiently fast.

The convergence of the point processes of the properly normalized eigenvalues in Section~\ref{sec:ppconvergence} yields a multitude of useful findings connected to the joint distribution of the eigenvalues. As an application, the structure of the eigenvectors of $\z\z'$ is explored in Section~\ref{sec:eigenvectors}.
Technical proofs are collected in Section~\ref{sec:proofs}.
Section~\ref{sec:autocov} is devoted to an extension of the results to the singular values of the 
{\em sample autocovariance matrices}  which are a generalization of the traditional autocovariance function for time series to high-dimensional matrices. In applications, the analysis of sample autocovariance matrices for different lags might help to detect dependencies in the data; see Lam and Yao \cite{lam:yao} for related work.
 We conclude with Appendix~\ref{appendix:A} which contains useful facts about regular variation and point processes.

%------------------------------------------------------------------------
\section{Preliminaries}\label{sec:2}\setcounter{equation}{0}
In this section we will discuss growth rates for $p=p_n\to\infty$ and introduce some notation. 
\subsection{Growth rates for $p$}
In many applications it is not realistic to assume 
that the dimension $p$ of the data and the sample size $n$ grow at the same rate, i.e., condition \eqref{eq:gamma} is unlikely to be
satisfied.
The aforementioned results of Soshnikov~\cite{soshnikov:2004,soshnikov:2006} and Auffinger et al.~\cite{auffinger:arous:peche:2009} already show that the value $\gamma$ in the growth
rate~\eqref{eq:gamma} does not appear in the \ds al limits. 
This obervation is in contrast to the light-tailed case; see \eqref{eq:MP} 
and \eqref{eq:geman}.  Davis et al.~\cite{davis:mikosch:pfaffel:2015,davis:pfaffel:stelzer:2014} allowed for more general rates for
$p_n\to\infty$ than linear growth in $n$. {\em However, they could not completely solve the technical difficulties arising with general growth rates of $p$.}
%Recall that $p=p_n\to\infty$ is the number of rows in the
%matrix $\z_n$. 
In what follows, we 
specify the growth rate of $(p_n)$:
%a non-degenerate limit distribution of the normalized eigenvalues of the sample autocovariance matrices. 
%To be precise, we assume
\begin{equation}\label{eq:p}
p=p_n=n^\beta \ell(n), \qquad n\ge1,\tag{$C_p(\beta)$}
\end{equation}
where $\ell$ is a slowly varying function and $\beta\ge 0$. If $\beta =0$, we also assume $\ell(n) \to \infty$. 
Condition~\ref{eq:p} is more general than the growth conditions in the literature; see 
\cite{auffinger:arous:peche:2009,davis:mikosch:pfaffel:2015,davis:pfaffel:stelzer:2014}. 

\subsection{Notation}\label{sec:section2.2}
Recall that $\z=\z_n= (Z_{it})_{i=1,\ldots,p;t=1,\ldots,n}$ is a $p\times n$ matrix with iid entries satisfying the regular variation 
condition \eqref{eq:regvar} for some $\alpha\in (0,4)$.
The sample covariance matrix $\z \z'$ has eigenvalues $\lambda_1, \ldots, \lambda_p$ whose order statistics were defined in \eqref{eq:order}. 

Important roles are played by the quantities $(Z_{it}^2)_{i=1,\ldots,p;t=1,\ldots,n}$ and their order statistics
 \begin{equation}\label{eq:zorder}
Z_{(1),np}^2 \ge Z_{(2),np}^2 \ge  \ldots \ge Z_{(np),np}^2, \qquad n,p\ge 1\,.
\end{equation}
As important are the row-sums 
\begin{equation}\label{eq:ll}
D_i^\rightarrow=D_i^{(n),\rightarrow}=\sum_{t=1}^n Z_{it}^2\,, \qquad
i=1,\ldots,p\,;\quad  n=1,2,\ldots\,,
\end{equation} with generic element $D^\rightarrow$ and their ordered values
\beam\label{eq:help6}
D_{(1)}^\rightarrow=D_{L_1}^\rightarrow\ge \cdots \ge D_{(p)}^\rightarrow=D_{L_p}^\rightarrow\,,
\eeam
where we assume without loss of generality that $(L_1,\ldots,L_p)$ is
a permutation of $(1,\ldots,p)$ for fixed~$n$.
\par
Finally, we introduce the column-sums
\beam\label{eq:colsum}
D_t^\downarrow=D_t^{(n),\downarrow}= \sum_{i=1}^p Z_{it}^2\,,\qquad t=1,\ldots,n\,; \quad p=1,2,\ldots\,,
\eeam
with generic element $D^\downarrow$ and we also adapt the notation from \eqref{eq:help6} to these quantities.

\subsubsection*{Norms} 
%\subsubsection*{Vector norm} 
For any $p$-dimensional vector $\bfv$, $\ltwonorm{\bfv}$ denotes its Euclidean norm.
%In this paper we will exclusively use the {\em $\ell_2$-norm}:
%\beao
%\ltwonorm{v}= \Big( \sum_{i=1}^p |v_i|^2\Big)^{1/2}\,.
%\eeao
For any $p\times p$ matrix $\bfC$, we write $\la_i(\bfC)$ for its $p$ singular values and we denote their order statistics by
\beao
\la_{(1)}(\bfC) \ge\cdots \ge \la_{(p)}(\bfC)\,. 
\eeao 
For any $p\times n $ matrix $\A=(a_{ij})$, we will use the 
{\em spectral norm}
$\|\A\|_2=\sqrt{\la_{(1)}(\A\A')}$, the 
{\em Frobenius norm} 
$
\frobnorm{\A}= \Big( \sum_{i=1}^p \sum_{j=1}^n |a_{ij}|^2\Big)^{1/2}
$ and the 
{\em max-row sum norm}
$
\inftynorm{\bfA}= \max_{i=1,\ldots,p} \sum_{j=1}^n |a_{ij}|\,.
$
%\item
%{\em Max-column sum norm:}
%\beao
%\|{A}\|_1 = \max_{j=1,\ldots,n} \sum_{i=1}^p |a_{ij}|\,.
%\eeao
%\end{itemize}
%We will frequently make use of the bound $\|\A\|_2\le \|\A\|_ F$. Standard references for matrix norms are \cite{belitski,bhatia:1997, horn, shores}. 

%-------------------------\subsubsection*{Eigenvectors} 
%We introduce the unit eigenvectors
%\begin{equation}\label{eq:eigenv}
%v_j^{\rightarrow}\,, \quad j=1,\ldots,p,
%\end{equation}
%of the sample covariance matrix $\z\z'$, i.e., $\ltwonorm{v_j}=1$, where $v_j^{\rightarrow}$ is associated with the $j$th largest eigenvalue $\lambda_{(j)}$. Furthermore, let 
%\begin{equation}\label{eq:eigenvd}
%v_j^{\downarrow}\,, \quad j=1,\ldots,n,
%\end{equation}
%denote the unit eigenvectors of the $n\times n$ matrix $\z'\z$.
%The canonical basis vectors are denoted by $\bfe_j$, where $(\bfe_j)_i=\delta_{ij}$ and $\delta_{ij}$ denotes the Kronecker-$\delta$.
%-----------------------------------------------
\section{Main results}\label{sec:mainresults}\setcounter{equation}{0}
\subsection{Basic approximations}
We commence with some basic approximation results for the eigenvalues and eigenvectors of $\z\z'$.
The approximating quantities have a simple structure and their \asy\ behavior is inherited by the eigenvalues
and has influence on the eigenvectors.
\begin{theorem}\label{thm:iidmain}
Consider a $p\times n$-dimensional matrix $\z$ with iid entries. We assume the following conditions:
\begin{itemize} \item
The \regvar\ condition \eqref{eq:regvar} for some  
$\alpha \in (0,4)$. 
\item $\E [Z]=0$ for $\alpha \ge 2$.
\item
The integer \seq\ $(p_n)$ 
has growth rate \ref{eq:p} for some  $\beta\ge 0$.
\end{itemize}
Then the following statements hold:
\begin{enumerate}
\item If $\beta\in [0,1]$, then
\beam\label{eq:rowa}
a_{np}^{-2}\,\max_{i=1,\ldots,p}\big|\la_{(i)}-\Dr_{(i)}\big|\stp 0\,.
\eeam
%and  for the eigenvectors we have
%\beao
%\ltwonorm{v_j^{\rightarrow} - \bfe_{L_j}} \cip 0\,, \qquad j=1,\ldots,p\,.
%\eeao
\item If $\beta > 1$, then
\beam\label{eq:cola}
a_{np}^{-2}\,\max_{i=1,\ldots,n}\big|\la_{(i)}-\Dc_{(i)}\big|\stp 0\,.
\eeam
\item If  $\min(\beta,\beta^{-1} ) \in ((\alpha/2-1)_+,1]$, then
\beam\label{eq:eigena}
a_{np}^{-2}\,\max_{i=1,\ldots,p}\big|\la_{(i)}-Z_{(i),np}^2\big|\stp 0\,.
\eeam
\end{enumerate}
\end{theorem}
\begin{remark}
\rm In \eqref{eq:cola} we have chosen to take maxima over the index set $\{1,\ldots,n\}$. We notice that
$\la_{(i)}=0$ for $i=p\wedge n+1,\ldots,p\vee n$. This is due to the fact that the $p\times p$ matrix $\bfZ\bfZ'$ and the $n\times n$ matrix
$\bfZ'\bfZ$ have the same positive eigenvalues. Moreover, for $n$ sufficiently large, $p\wedge n=p$ for $\beta\in (0,1)$ and $p\wedge n=n$ for $\beta>1$, i.e., only in the case 
$\beta=1$ both cases $n\le p$ or $p\le n$ are possible.
\end{remark} 
\begin{remark}\rm 
The condition $\min(\beta,\beta^{-1} ) \in ((\alpha/2-1)_+,1]$ in part (3) is only a restriction when $\alpha\in (2,4)$. We notice 
that this condition implies $(n\vee p)/a_{np}^2\to 0$. In turn, this means that centering of the quantities $a_{np}^{-2}D_i^\rightarrow$ and $a_{np}^{-2}D_i^\downarrow$ in the limit theorems 
can be avoided.
This argument is relevant in various parts of the proofs.
\end{remark}
\begin{remark}\rm
In Figure~\ref{fig:lambda_comparison} we illustrate 
the different approximations of the eigenvalues $(\la_{(i)})$ by 
$(\Dr_{(i)})$ as suggested by \eqref{eq:rowa}
and $(Z_{(i),np}^2)$ as suggested by \eqref{eq:eigena}. For $Z$ we choose the density
\beam\label{eq:distrsim}
f_Z(x) =
\left\{\begin{array}{cc}
 \frac{\alpha}{(4|x|)^{\alpha + 1}}\,, & \mbox{if } |x| > 1/4 \\
1\,, & \mbox{otherwise.}
\end{array}\right.
\eeam
In the left graph, we focus on the largest eigenvalue $\la_{(1)}$. We show smoothed histograms of the approximation errors
$a_{np}^{-2}(\lambda_{(1)}-D_{(1)}^\rightarrow)$, $a_{np}^{-2}(\lambda_{(1)}-Z_{(1),np}^2)$.
By Cauchy's interlacing theorem (see \cite[Lemma~22]{tao:vu:2012}), the considered differences
are non-negative. 

In the right graph, we take the maxima as in \eqref{eq:rowa} and \eqref{eq:eigena} and show smoothed histograms of the approximation errors
$a_{np}^{-2}\max_{i\le p} |\lambda_{(i)}-D_{(i)}^\rightarrow|$, $a_{np}^{-2}\max_{i\le p}|\lambda_{(i)}-Z_{(i),np}^2|$. We take absolute values to deal with negative differences. 
Figure~\ref{fig:lambda_comparison} indicates 
that $(D_{(i)}^\rightarrow)$ yield a much better approximation to $(\la_{(i)})$ than $(Z_{(i),np}^2)$. Notice the different scaling on the $x$- and $y$-axes.
\begin{figure}[htb!]
  \centering
  \subfigure{
    \includegraphics[scale=0.4]{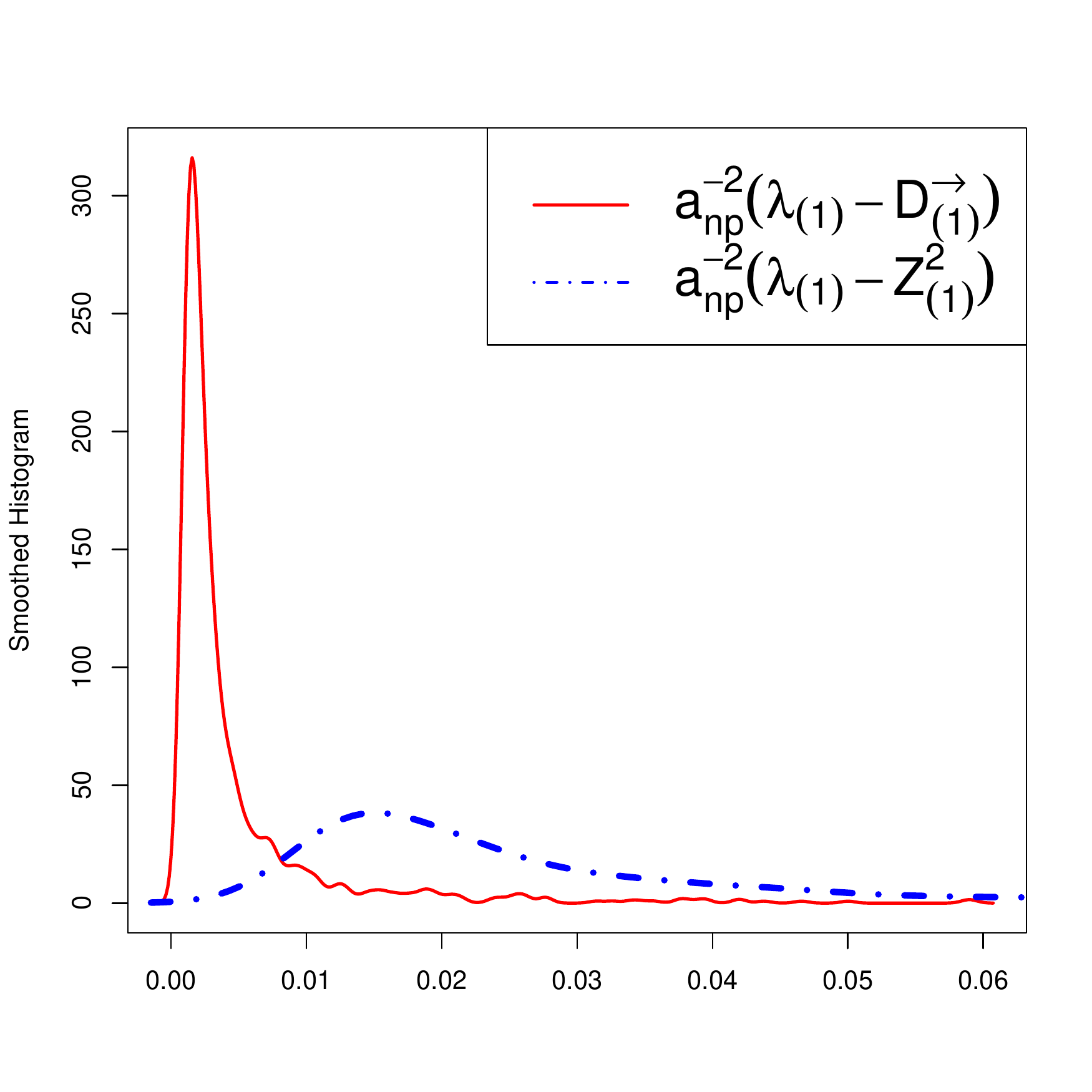}
  }
  \subfigure{
    \includegraphics[scale=0.4]{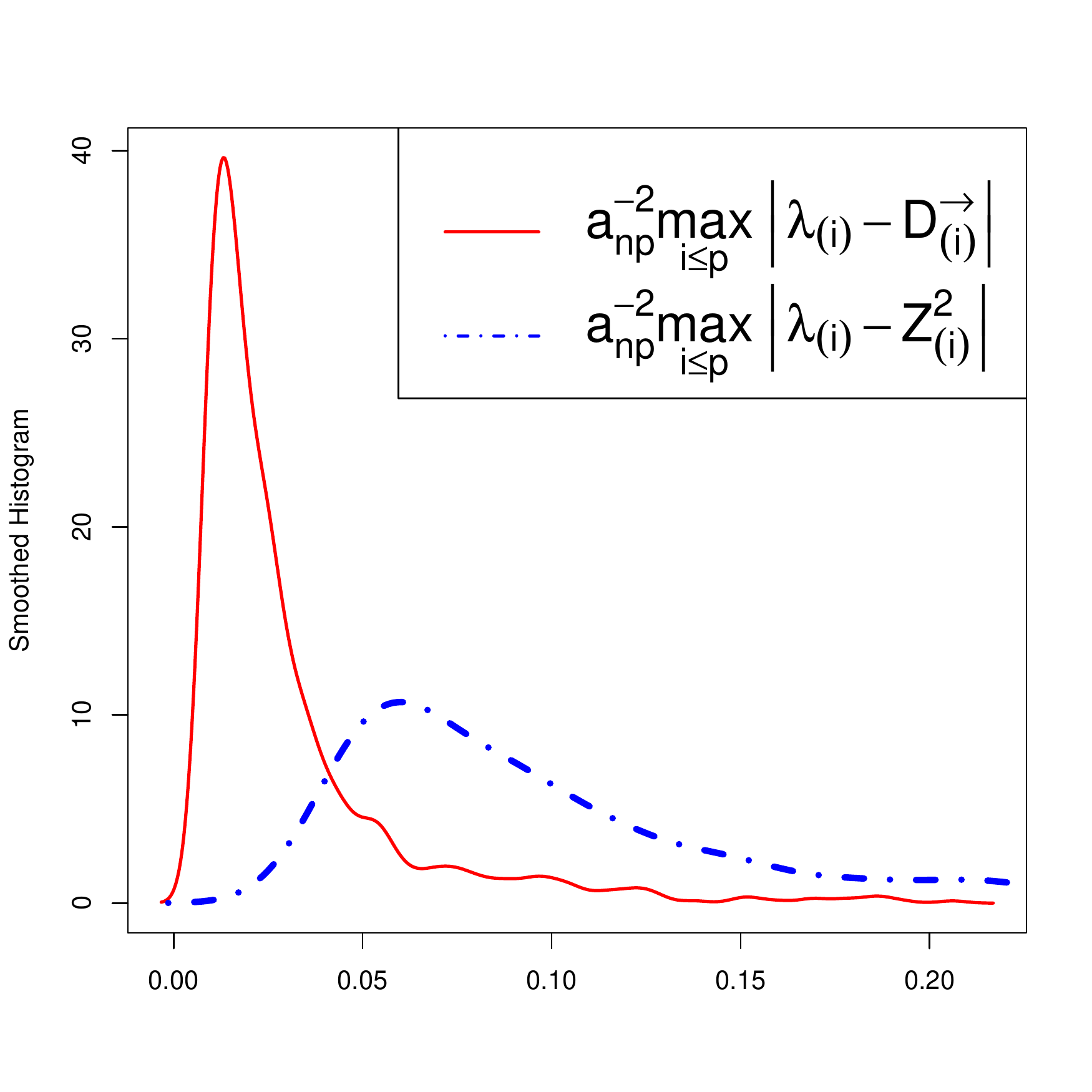}
  }
  \caption{Smoothed histograms of the  approximation errors for the normalized eigenvalues $(a_{np}^{-2}\la_{(i)})$
for entries $Z_{it}$ with density \eqref{eq:distrsim}, $\alpha=1.6$, $\beta=1$, $n=1,000$ and $p=200$.
}
  \label{fig:lambda_comparison}
\end{figure}
\end{remark}
\par 
The proof of Theorem~\ref{thm:iidmain} will be given in Section~\ref{sec:proofs}. A main step in the proof is provided by the following
result whose proof will also be given in Section~\ref{sec:proofs};
a version of this theorem was proved in Davis et al.~\cite{davis:mikosch:pfaffel:2015} 
under more restrictive conditions on the growth rate of $(p_n)$.

\begin{theorem}\label{prop:offdiagonal}
Assume the conditions of Theorem~\ref{thm:iidmain} on  $\z$ and $(p_n)$.
\begin{enumerate}
\item
If $\beta\in [0,1]$ we have
%has iid entries satisfying the \regvar\ condition \eqref{eq:regvar} for some  
%$\alpha \in (0,4)$. If $\E[|Z|]<\infty$ we also suppose that $\E [Z]=0$. Then for any \seq\ $(p_n)$ 
%satisfying \ref{eq:p} with $\beta\in [0,1]$ 
\beao
a_{np}^{-2} \twonorm{\bfZ \bfZ' - \diag(\bfZ \bfZ')} \cip 0\,,\qquad\nto\,.
\eeao
\item
If $\beta\ge 1$ we have 
\beao
a_{np}^{-2} \twonorm{\bfZ' \bfZ - \diag(\bfZ' \bfZ)} \cip 0\,,\qquad\nto\,.
\eeao
\end{enumerate}
\end{theorem}
The second part of this theorem follows from the first one by an interchange of $n$ and $p$.
%observing that the $p \times p$ matrix $\z\z'$ and the $n \times n$ matrix $\z'\z$ have the same non-zero eigenvalues; see Bhatia \cite[p.~64]{bhatia:1997}. 
Indeed, if $\beta\ge 1$, we can write $n=p^{1/\beta}\ell(p)$ for some \slvary\
\fct\ $\ell$ and then part (2) follows from part (1).
\begin{remark}
Theorem \ref{prop:offdiagonal} shows that the largest eigenvalues of $\z\z'$ are determined by the largest diagonal entries. 
In the case of heavy-tailed Wigner matrices, however, the diagonal elements do not play any particular role.
\end{remark}

From this theorem one immediately obtains 
a result about the approximation of the eigenvalues
 of $\z\z'$ and $\z'\z$ by those of $\diag(\z\z')$ and $\diag(\z'\z)$, respectively. Indeed, 
for any symmetric $p\times p$ matrices $\bfA,\bfB$, by {\em Weyl's inequality} (see Bhatia \cite{bhatia:1997}),
\beam\label{eq:weylin}
\max_{i=1,\ldots,p}\big|\la_{(i)}(\bfA+\bfB)-\la_{(i)}(\bfA)\big|\le \|\bfB\|_2\,.
\eeam
If we now choose 
$\bfA+\bfB=\z\z'$ and $\bfA= \diag (\z\z')$ (or $\bfA+\bfB=\z'\z$ and $\bfA= \diag (\z'\z)$) we obtain the following result.

\begin{corollary}\label{cor:687}
Assume the conditions of Theorem~\ref{thm:iidmain} on  $\z$ and $(p_n)$.
\begin{enumerate}
\item
If $\beta\in [0,1]$ we have
\beao
a_{np}^{-2}\,\max_{i=1,\ldots,p}\big|\la_{(i)}-\la_{(i)}(\diag(\z\z'))\big|\stp 0\,,\qquad \nto\,.
\eeao
\item
If $\beta> 1$  we have
\beao
a_{np}^{-2}\,
\max_{i=1,\ldots,n}\big|\la_{(i)}-\la_{(i)}(\diag(\z'\z))\big|
\stp 0\,,\qquad\nto \,.
\eeao
\end{enumerate}
\end{corollary}
Now \eqref{eq:rowa} and \eqref{eq:cola} are immediate con\seq s of this corollary. Indeed, we have
$\la_{(i)}(\diag(\z\z'))=D_{(i)}^\rightarrow$ and $\la_{(i)}(\z'\z)=D_{(i)}^\downarrow$, $i=1,\ldots,p\wedge n$.
%\begin{proposition}\label{prop:offdiagonal1}
%Assume that $\z=\z_n$ has iid entries satisfying the \regvar\ condition \eqref{eq:regvar} for some  
%$\alpha \in (0,4)$. If $\E[|Z|]<\infty$ we also suppose that $\E [Z]=0$. Then for any \seq\ $(p_n)$ 
%satisfying \ref{eq:p} with $\beta>1$ we have
%\beao
%a_{np}^{-2} \twonorm{\z_n'\z_n - \diag(\z_n'\z_n)}\stp 0\,,\qquad\nto\,.
%\eeao
%\end{proposition}
%Note that for $\beta>1$ we have $\lim_{\nto} p/n= \infty$. This means that 
%$\z_n'\z_n$ has  asymptotically a much smaller dimension than $\z_n\z_n'$ and therefore it is more convenient to work with $\z_n'\z_n$ 
%when bounding the spectral norm. 
%\begin{corollary}\label{cor:6871}
%Under the conditions of Proposition~\ref{prop:offdiagonal1},
%\beao
%a_{np}^{-2}\,\max_{i=1,\ldots,n}\big|\la_{(i)}-\la_{(i)}(\diag(\z_n'\z_n))\big|\stp 0\,.
%\eeao
%\end{corollary}

\subsection{Point process \con }\label{sec:ppconvergence}
In this section we want to illustrate how the approximations from Theorem~\ref{thm:iidmain} can be used
to derive \asy\ theory for the largest eigenvalues of $\z\z'$ via the weak \con\ of suitable \pp es.
The limiting \pp\ involves the points of the Poisson process
\begin{equation}\label{eq:N}
N_\Gamma=\sum_{i=1}^\infty \vep_{\Gamma_i^{-2/\alpha}}\,,\qquad \nto\,,
\end{equation}
where $\vep_y$ is the Dirac measure at $y$, 
\beao%\label{eq:Gamma}
\Gamma_i=E_1+\cdots + E_i\,,\qquad i\ge 1\,,
\eeao
and $(E_i)$ is a sequence of iid standard exponential random variables. In other words, $N_\Gamma$ is a Poisson point process on $(0,\infty)$ with mean \ms\
$\mu(x,\infty)= x^{-\alpha/2}$, $x>0$.
\par
\begin{lemma}\label{lem:pp}
Assume the conditions of Theorem~\ref{prop:offdiagonal} hold. 
\begin{enumerate}
\item
If $\beta\ge 0$, then
\beam\label{eq:immb}
\sum_{i=1}^p \vep_{a_{np}^{-2}(D_i^\rightarrow-c_n)}\cid N_\Gamma\,,\qquad \nto\,,
\eeam  
where $c_n=0$ if $\E[D^\rightarrow]=\infty$ and $c_n=\E[D^\rightarrow]=n\,\E[Z^2]$ otherwise. 
\item
If $\beta\ge 0$, then
\beam\label{eq:imma}
\sum_{i=1}^p \vep_{a_{np}^{-2}Z_{(i),np}^2}\cid N_\Gamma\,,\qquad \nto\,,
\eeam  
\end{enumerate}
The weak \con\ of the \pp es holds in the space of 
point \ms s with state space $(0,\infty)$ equipped with the vague topology; see Resnick \cite{resnick:2007}.
\end{lemma}
\begin{remark} \rm
Similar results were used in the proofs of Davis et al.~\cite{davis:mikosch:heiny:xie:2015,davis:mikosch:pfaffel:2015}.
We also mention that the centering $c_n$ in the finite variance case can be avoided if $n/a_{np}^2\to 0$. 
The latter condition is satisfied if $\beta>\alpha/2-1$.
\end{remark}
\begin{proof} 
Part (1) follows from Lemma~\ref{lem:ppr}. As regards part (2), we observe that
\beam\label{eq:gp}
\sum_{i=1}^p\sum_{t=1}^n \vep_{a_{np}^{-2} Z_{it}^2}\std N_\Gamma\,;
\eeam
see e.g. Resnick~\cite{resnick:1987}, Proposition 3.21. On the other hand, $a_{np}^{-2} Z_{(p),np}^2\stp 0$ which together with
\eqref{eq:gp} yields part (2).
\end{proof}

Theorem~\ref{thm:iidmain}  and arguments similar to the proofs in Davis et al.~\cite{davis:mikosch:heiny:xie:2015,davis:mikosch:pfaffel:2015}
enable one to derive the weak \con\ of
the point processes of the normalized eigenvalues. 
%\begin{equation}\label{eq:ppdef}
%N_n=\sum_{i=1}^p \vep_{a_{np}^{-2} \lambda_i}, \quad \widetilde{N}_n^{\rightarrow}=\sum_{i=1}^p \vep_{a_{np}^{-2} (\lambda_i-n\,\E[Z^2])}, \qquad \mbox{and} \qquad
%\widetilde{N}_n^{\downarrow}=\sum_{i=1}^n \vep_{a_{np}^{-2} (\lambda_i-p\,\E[Z^2])}, \qquad n\ge 1\,.
%\end{equation}
\begin{theorem}\label{cor:1}
Assume the conditions of Theorem~\ref{thm:iidmain}. If
$\min(\beta,\beta^{-1} ) \in ((\alpha/2-1)_+,1]$ 
then
\beam\label{eq:imm}
\sum_{i=1}^p \vep_{a_{np}^{-2}\la_i}\std N_\Gamma\,,
\eeam 
in the space of point measures 
with state space $(0,\infty)$ equipped with the vague topology. 
\end{theorem}
\begin{proof}
The limit relation \eqref{eq:imm} follows from \eqref{eq:imma} in combination with
\eqref{eq:eigena}. Alternatively, one can exploit \eqref{eq:immb} both for $(D_i^\rightarrow)$ and 
$(D_t^\downarrow)$ (notice that the \pp\ \con\ for the latter \seq\ follows by interchanging the roles of $n$ and $p$), the fact that
$(n\vee p)/a_{np}^2\to 0$ if $\min(\beta,\beta^{-1})\in ((\alpha/2-1)_+,1]$ (hence centering of the points $(D_i^\rightarrow)$ and  $(D_t^\downarrow)$ in \eqref{eq:immb}
can be avoided for $\E[Z^2]<\infty$) and finally using the approximations \eqref{eq:rowa} or \eqref{eq:cola}.
\end{proof}
The weak \con\ of the \pp es of the normalized eigenvalues of $\bfZ\bfZ'$ 
in Theorem~\ref{cor:1} allows one to use the conventional tools in this
field; see Resnick~\cite{resnick:2007,resnick:1987}. An immediate con\seq\ is
\beam \label{eq:helpgamma1}
a_{np}^{-2}\big(\la_{(1)},\ldots,\la_{(k)}\big)\std \big(\Gamma_1^{-2/\alpha},\ldots,\Gamma_k^{-2/\alpha}\big)
\eeam
for any fixed $k\ge 1$. Using the methods of Davis et al.~\cite{davis:mikosch:heiny:xie:2015} shows for $\alpha \in (2,4)$
\beam \label{eq:helpgamma2}
a_{np}^{-2}\big(\la_{(1)}-(p\vee n) \E[Z^2],\ldots,\la_{(k)}-(p\vee n) \E[Z^2]\big)\std \big(\Gamma_1^{-2/\alpha},\ldots,\Gamma_k^{-2/\alpha}\big)\,.
\eeam
Equations \eqref{eq:helpgamma1} and \eqref{eq:helpgamma2} yield that for $\alpha \in (0,4)$ and any fixed $k\ge 1$,
\beam \label{eq:helpgamma3}
a_{np}^{-2}\big(\la_{(1)}- \la_{(2)},\ldots, \la_{(k)}- \la_{(k+1)}\big)\std \big(\Gamma_1^{-2/\alpha} - \Gamma_{2}^{-2/\alpha},\ldots,\Gamma_k^{-2/\alpha} - \Gamma_{k+1}^{-2/\alpha} \big)\,.
\eeam
Related results can also be derived for an increasing number of order statistics, e.g.~the joint \con\ of the largest eigenvalue $a_{np}^{-2}\la_{(1)}$ and the trace $a_{np}^{-2}(\la_1+\cdots+\la_p)$. In particular, one obtains for
$\alpha\in (0,2)$ under the conditions of Theorem~\ref{cor:1} that
\beao
\dfrac{\la_{(1)}}{\la_1+\cdots +\la_p} \std \dfrac{\Gamma_1^{-2/\alpha}}{\Gamma_1^{-2/\alpha}+\Gamma_2^{-2/\alpha}+\cdots}\,.
\eeao 
We refer to Davis et al.~\cite{davis:mikosch:pfaffel:2015} for details on the proofs and more examples.

In the next subsection we will show how the above results on the joint convergence of eigenvalues can be applied to approximate the eigenvectors of $\z\z'$.

%---------------------------------------------------------------------
\subsection{Eigenvectors}\label{sec:eigenvectors}
In this section we assume the conditions of Theorem~\ref{prop:offdiagonal} and $\beta \in [0,1]$.
From Theorem~\ref{prop:offdiagonal}(1) %and the proof of Corollary~\ref{cor:688} 
we know that $\z\z'$ is approximated in spectral norm by $\diag(\z\z')$.
% and $\Y_n$, both of which are diagonal matrices, are good approximations to the sample covariance matrix $\z_n\z_n'$ in the spectral norm. 
The unit eigenvectors %$\bfv_j$, $j=1,\dots,p$, 
of a $p\times p$ diagonal matrix are the canonical basis vectors $\bfe_j\in\R^p$, $j=1,\dots,p$. 
This raises the question as to whether $(\bfe_j)$ are good approximations of the eigenvectors $(\bfv_j)$ of $\z\z'$.
By $\bfv_{{j}}$ we denote the unit eigenvector associated with the $j$th largest eigenvalue $\la_{(j)}$. The unit eigenvector associated with the $j$th largest eigenvalue of $\diag(\z\z')$ is $\bfe_{L_j}$, where $L_j$ is defined in \eqref{eq:help6}. Our guess that $\bfv_j$ is approximated by $\bfe_{L_j}$ is confirmed by the following result.

\begin{theorem}\label{thm:eigenvec} Assume the conditions of Theorem~\ref{thm:iidmain} and let
$\beta\in [0,1]$. Then for any fixed $k\ge 1$,
\beao
\ltwonorm{\bfv_k - \bfe_{L_k}} \cip 0\,, \quad \nto \,.
\eeao
\end{theorem}

Indeed, $\bfv_j$ and $\bfe_{L_j}$ share another property: they are {\em localized} which means that they are concentrated only in a few components. Vectors which are not localized are called {\em delocalized}.
Figure~\ref{fig:eigenvectors} shows the outcome of a simulation example in which we visualize 
the components of the unit eigenvector associated with the largest eigenvalue of $\z\z'$ 
for a simulated data matrix $\z$ with iid Pareto(0.8) entries. In the right graph
we see that only one of the $p=200$ components is significant. 
Hence we can find a canonical basis vector $\bfe_{k}$ such that $\ltwonorm{\bfe_{k}-\bfv_{1}}$ is small. 
Therefore the eigenvector is localized. This is in stark contrast to the case of iid standard normal entries; see the left graph.
%When the entry distribution is standard normal, the picture looks completely different. 
Then many components are of similar magnitude, hence the eigenvector is delocalized. 
Typically, the eigenvectors tend to be localized when the entry distribution has an infinite fourth moment, 
while they tend to be delocalized  otherwise; see Benaych-Georges and P\'{e}ch\'{e} \cite{benaych:peche} for the case of Wigner matrices. 
\begin{figure}[htb!]
  \centering
  \subfigure%[Standard normal entries]
{
    \includegraphics[scale=0.4]{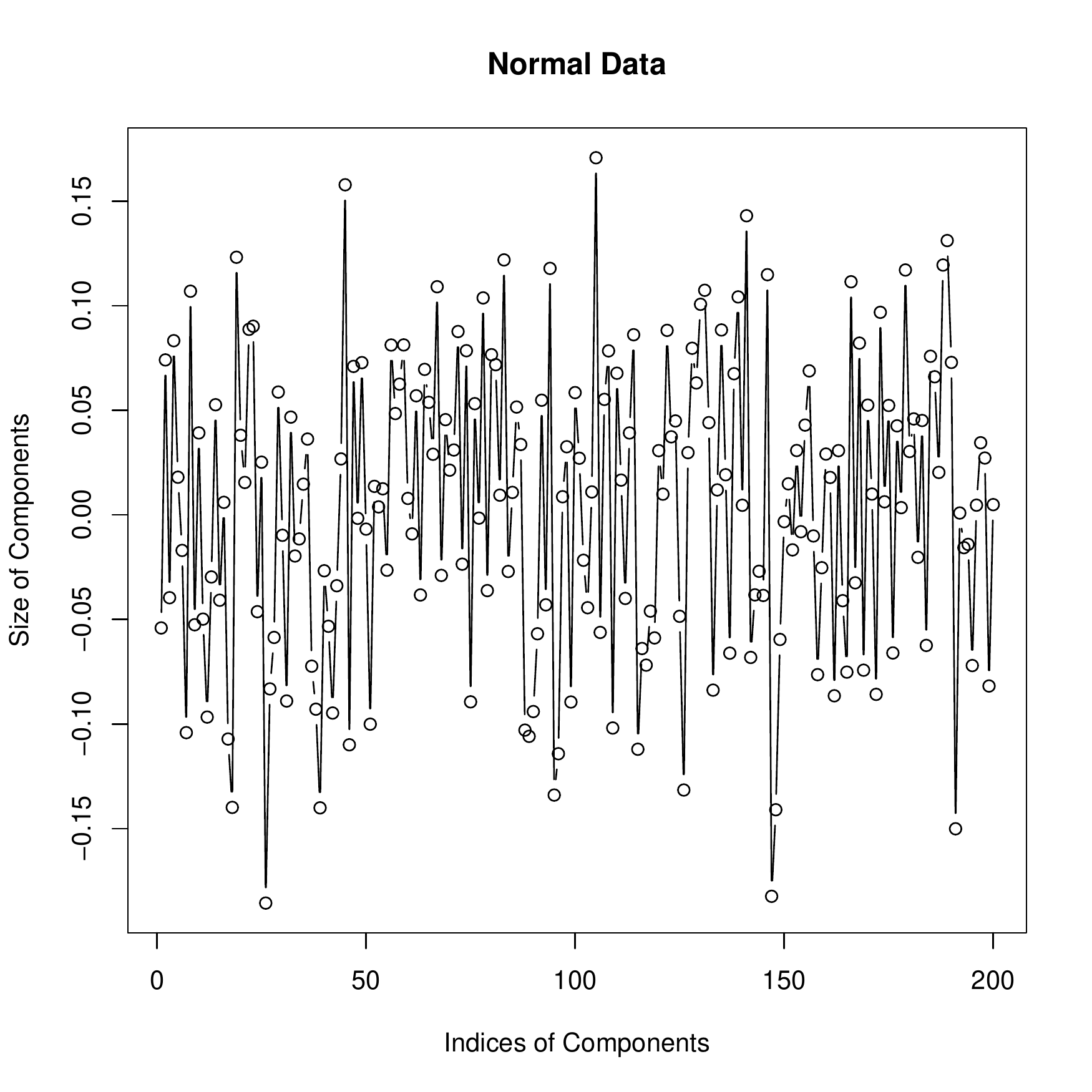}
  }
%  \subfigure[Pareto(0.8) entries]
{
    \includegraphics[scale=0.4]{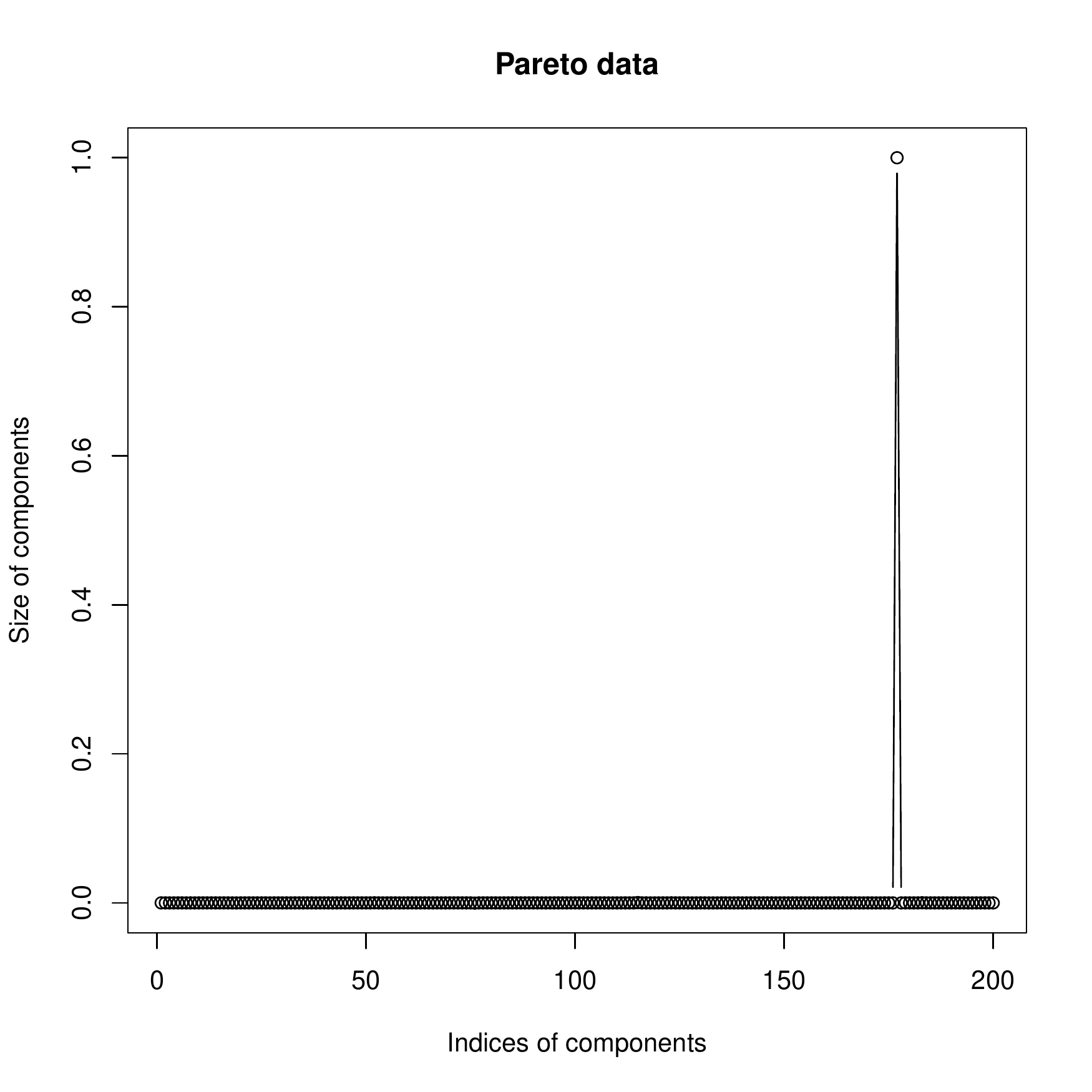}
  }
  \caption{The components of the eigenvector $\bfv_1$. Right: The case of iid Pareto(0.8) entries. Left: The case
of iid standard normal entries. 
%$ \max_{i=1,\dots, p} |V_{1, i}| = 1 - 10^{-5}???$ 
We choose $p=200$ and $n=1,000$.
}
  \label{fig:eigenvectors}
\end{figure}

%The observation from Figure~\ref{fig:eigenvectors} that the eigenvector appears to be localized 
%in a single simulation study does not mean that that the eigenvectors are ``close'' to the canonical basis 
%vectors in general. Also if many of the largest $k$ eigenvalues lie in a very small interval (as for 
%standard normal data) it is not easy to determine which basis vector should provide a good approximation to which eigenvector even if we assume the set of eigenvectors and the set of basis vectors to coincide. In applications small approximation errors might thus lead to grave mistakes. The distributional limit of the spectral gaps
%\begin{equation*}
%a_{np}^{-2} (\lambda_{(i)}- \lambda_{(i+1)})\,, \quad i=1,\ldots, k,
%\end{equation*}
%will provide a measure of the placing of $(\lambda_{(i)})$ on the positive real line. From \eqref{eq:jointconv} we conclude that
%\beam\label{eq:gap}
%a_{np}^{-2} (\lambda_{(i)}- \lambda_{(i+1)})\cid \Gamma_i^{-2/\alpha}-\Gamma_{i+1}^{-2/\alpha}\,, \quad i=1,\ldots, k.
%\eeam

\begin{proof}[Proof of Theorem~\ref{thm:eigenvec}]
Fix $k\ge 1$. Since $p\to \infty$ we can assume $k\le p$ for sufficiently large $n$.
We observe that
\begin{equation*}
\z\z' \,\bfe_j -\Dr_j \,\bfe_j = \Big( \sum_{t=1}^n Z_{1t}Z_{jt}, \ldots,\sum_{t=1}^n Z_{j-1,t}Z_{jt}, 0,  \sum_{t=1}^n Z_{j+1,t}Z_{jt}, \ldots,\sum_{t=1}^n Z_{pt}Z_{jt} \Big)'\,, \quad j=1,\ldots,p\,,
\end{equation*}
are the columns of $\z\z' -\diag(\z\z')$. By Theorem~\ref{prop:offdiagonal}(1),
\beam \label{eq:colas}
a_{np}^{-2} \max_{j=1,\ldots,p} \ltwonorm{\z\z' \bfe_j -\Dr_j \bfe_j}
 \le a_{np}^{-2} \twonorm{\bfZ \bfZ' - \diag(\bfZ \bfZ')} \cip 0\,,\qquad\nto\,.
\eeam

If we set $\bfH^{(n)}=a_{np}^{-2}\z\z'$, $\bfv^{(n)}=\bfe_{L_k} \in \R^p$ and $\lambda^{(n)}= a_{np}^{-2} \Dr_{L_k}$, we see that
\begin{equation*}
a_{np}^{-2}\z\z' \bfe_{L_k} = a_{np}^{-2} \Dr_{L_k} \bfe_{L_k} + \vep^{(n)} \bfw^{(n)},
\end{equation*}
 where $\bfw^{(n)}= \ltwonorm{\z\z' \bfe_{L_k} -\Dr_{L_k} \bfe_{L_k}}^{-1} (\z\z' \bfe_{L_k} -\Dr_{L_k} \bfe_{L_k})$ is a unit vector and $\vep^{(n)}= a_{np}^{-2} \ltwonorm{\z\z' \bfe_{L_k} -\Dr_{L_k} \bfe_{L_k}} \cip 0$ by \eqref{eq:colas}. 

Before we can apply Proposition~\ref{prop:perturbation} we need to show that with probability converging to $1$, there are no other eigenvalues in a suitably small interval around $\lambda_{(k)}$. Let $s>1$. We define the set
\begin{equation*}
\Omega_n = \Omega_n(k,s)= \{a_{np}^{-2} |\la_{(k)}-\la_{(i)}|>s\, \vep^{(n)}\, :\, i\neq k =1,\ldots,p  \}\,.
\end{equation*} 
From \eqref{eq:colas} we get $s\, \vep^{(n)}\to 0$. Then using this and \eqref{eq:helpgamma3}, we obtain
\begin{equation*}
\begin{split}
\lim_{\nto} \P(\Omega_n^c) &= \lim_{\nto} \P( a_{np}^{-2} \min\{ \la_{(k-1)}-\la_{(k)}, \la_{(k)}-\la_{(k+1)}\} \le s\, \vep^{(n)}) =0
\end{split}
\end{equation*}

By Proposition~\ref{prop:perturbation} the unit eigenvector $\bfv_k$ associated with $\la_{(k)}$ and the projected vector $\bfP_{\bfe_{L_k}}(\bfv_{k})=(\bfv_{k})_{L_k}\bfe_{L_k}$ satisfy for fixed $\delta>0$:
\begin{equation*}
\begin{split}
\limsup_{\nto} \P(\ltwonorm{\bfv_k-(\bfv_{k})_{L_k}\bfe_{L_k}}>\delta) &
\le \limsup_{\nto} \P(\{\ltwonorm{\bfv_k-(\bfv_{k})_{L_k}\bfe_{L_k}}>\delta\} \cap \Omega_n)+ \limsup_{\nto}\P(\Omega_n^c)\\
&\le \limsup_{\nto} \P(\{2\vep^{(n)}/(s\,\vep^{(n)}-\vep^{(n)})>\delta\} \cap \Omega_n)\\
&\le \limsup_{\nto} \P(\{2/(s-1)>\delta\})=  \1_{\{ 2/(s-1)>\delta\}}.
\end{split}
\end{equation*}
The right-hand side is zero for sufficiently large $s$. Since both $\bfv_k$ and $\bfe_{L_k}$ are unit vectors this means that
\beao
\ltwonorm{\bfv_k - \bfe_{L_k}} \cip 0\,, \quad \nto \,.
\eeao
This proves our result on eigenvectors.
\end{proof}

%-------------------------------------------------------------------------
%Under more restrictive conditions on the dimension $p$, Davis et al.~\cite{davis:pfaffel:stelzer:2014} showed that the eigenvalues of $\z\z'$ are basically determined by the diagonal entries of this matrix if 
%the norm of  $a_{np}^{-2} (\bfZ_n \bfZ_n' - \diag(\bfZ_n \bfZ_n'))$ is ``small''. Here
%\begin{equation*}
%(\bfZ_n \bfZ_n' - \diag(\bfZ_n \bfZ_n'))_{ij}= \sum_{t=1}^n Z_{it} Z_{jt}, \quad i\neq j =1, \ldots,p.
%\end{equation*}
%Under more restrictive conditions on the dimension $p$,  showed the fo
%that the eigenvalues of $\z\z'$ are basically determined by the diagonal entries of this matrix if 
%the norm of  $a_{np}^{-2} (\bfZ_n \bfZ_n' - \diag(\bfZ_n \bfZ_n'))$ is ``small''. Here
\section{Proof of Theorem~\ref{thm:iidmain} }\label{sec:proofs}
{\em In what follows, $c$ stands for any constant whose value is not of interest.}
We write $(Z_t)$ for an iid \seq\ with the same \ds\ as $Z$.
\par
The plan of the proof is as follows:
\begin{enumerate}
\item
We prove  Theorem~\ref{prop:offdiagonal} which implies \eqref{eq:rowa} and \eqref{eq:cola}; see Corollary~\ref{cor:687}.
In view of the arguments after Theorem~\ref{thm:iidmain} it suffices to consider only the case $\beta\in [0,1]$.
\item
We prove \eqref{eq:eigena}.
\end{enumerate} 
\subsection{Proof of Theorem~\ref{prop:offdiagonal}}
We proceed in several steps.
\begin{proof}[The case $\alpha \in (0,8/3)$]
If $\alpha\in [1,2)$ and $\E[|Z|]< \infty$, we have
\begin{equation*}
a_{np}^{-1} \twonorm{\bfZ - (\bfZ -\E[\bfZ])}= |\E[Z]| \frac{\sqrt{np}}{a_{np}}\to 0\,, \quad \nto.
\end{equation*}
Therefore, without loss of generality $\E[Z]$ can be assumed $0$ in this case.

From now on we assume $\E[Z]=0$ whenever $\E[|Z|]$ exists. Since the Frobenius norm $\frobnorm{\cdot}$ is an upper bound of the spectral norm we have
\beao
\twonorm{\bfZ \bfZ' - \diag(\bfZ \bfZ')}^2&\le& \frobnorm{\bfZ \bfZ' - \diag(\bfZ \bfZ')}^2\\
&=&\sum_{i, j =1; i\neq j}^p   \sum_{t=1}^n Z_{it}^2 Z_{jt}^2+\sum_{i, j =1; i\neq j}^p   \sum_{t_1,t_2=1; t_1\neq t_2}^n Z_{i,t_1} Z_{j,t_1}Z_{i,t_2} Z_{j,t_2}\\
&=&\sum_{i, j =1; i\neq j}^p   \sum_{t=1}^n Z_{it}^2 Z_{jt}^2 \big[\1_{\{ Z_{it}^2 Z_{jt}^2  > a_{np}^4 \}}+\1_{\{ Z_{it}^2 Z_{jt}^2  \le a_{np}^4 \}}\big]
%\sum_{i, j =1; i\neq j}^p   \sum_{t=1}^n Z_{it}^2 Z_{jt}^2 \1_{\{ Z_{it}^2 Z_{jt}^2\le   a_{np}^4 \}}
+I_2^{(n)}\\
&=&I_{11}^{(n)}+I_{12}^{(n)}+I_2^{(n)}\,.
\eeao
Thus it suffices to show that each of the expressions on the \rhs\ when normalized with $a_{np}^4$ converges to zero in \pro y.
%\begin{equation*}
%\begin{split}
%\P(a_{np}^{-2} \twonorm{\bfZ_n \bfZ_n' &- \diag(\bfZ_n \bfZ_n')}>c) \le \P(a_{np}^{-2} \frobnorm{\bfZ_n \bfZ_n' - \diag(\bfZ_n \bfZ_n')}>c)\\
%&\le \P \Big( \sum_{i, j =1; i\neq j}^p  \Big( \sum_{t=1}^n Z_{it} Z_{jt}\Big)^2 >c a_{np}^4 \Big)\\
%&\le \P \Big( \sum_{i, j =1; i\neq j}^p   \sum_{t=1}^n Z_{it}^2 Z_{jt}^2 >c a_{np}^4\Big) + \P \Big( \sum_{i, j =1; i\neq j}^p   \sum_{t_1,t_2=1; t_1\neq t_2}^n Z_{i,t_1} Z_{j,t_1}Z_{i,t_2} Z_{j,t_2} >c a_{np}^4 \Big)\\
%&= P_1^{(n)}+P_2^{(n)}.
%\end{split}
%\end{equation*}
%We proceed by showing that $P_1^{(n)}$ and $P_2^{(n)}$ converge to zero. One sees
We have for any $\epsilon>0$,
\begin{equation*}
\P \big( I_{11}^{(n)}>\epsilon\,a_{np}^4\big)\le   p^2\,n\, \P(Z_{1}^2 Z_{2}^2 > a_{np}^4) \to 0\,.
\end{equation*}
%We obtain $P_{11}^{(n)} \le p^2n \P(Z_{1}^2 Z_{2}^2 > a_{np}^4) \to 0$ as $\nto$. 
Here we also used the fact that $Z_1Z_2$ is regularly varying with index $\alpha$; see Embrechts and Goldie \cite{embrechts:goldie:1980}. 
An application of Markov's inequality and Lyapunov's moment inequality with $\gamma\in (\alpha/2, 4/3)$ if $\alpha \in [2,8/3)$ and $\gamma=1$ otherwise shows that 
\begin{equation*}
\P\big( I_{12}^{(n)}>\epsilon\,a_{np}^4\big)\le
c\, \frac{p^2n}{a_{np}^4} \Big(\E[|Z_{1} Z_{2}|^{2\gamma}\1_{\{ |Z_{1} Z_{2}| \le a_{np}^2 \}}]\Big)^{\frac{1}{\gamma}}\le 
c\, p^{2-\frac{2}{\gamma}}\,n^{1-\frac{2}{\gamma}+\delta}\to 0,
\end{equation*} 
where we used Karamata's theorem (see Bingham et al. \cite{bingham:goldie:teugels:1987}), and the constant $\delta>0$ can be chosen arbitrarily small due to the Potter bounds.
\par
In the case $\alpha \in (0,2)$ the probability $P_{2}^{(n)}=\P(I_2^{(n)}>\epsilon\,a_{np}^4)$ can be handled analogously. Next, we turn to $P_{2}^{(n)}$ in the case $\alpha \in (2,8/3)$. In particular, $\E [Z^2]<\infty$. With \v Cebychev's inequality, also using the fact that $\E [Z]=0$,  we find that
\begin{equation}\label{eq:sdvfg}
P_{2}^{(n)} \le c \frac{1}{a_{np}^8} \E\Big[ \Big( \sum_{i, j =1; i\neq j}^p   \sum_{t_1,t_2=1; t_1\neq t_2}^n Z_{i,t_1} Z_{j,t_1}Z_{i,t_2} Z_{j,t_2}\Big)^2 \Big]\le c\,\frac{(p\,n)^2}{a_{np}^8} \to 0.
\end{equation}
The case $\alpha=2$ is
most difficult because the second moment of $Z$ can be infinite. 
Without loss of generality we assume that $Z$ is continuous. Otherwise, we add independent centered normal random variables 
to each of the entries $Z_{it}$; due the normalization $a_{np}^2$ 
the \asy\ properties of the eigenvalues remain the same, i.e., the added normal components are \asy ally negligible.  
In view of Hult and Samorodnitsky \cite[Lemma~$4.2$]{hult:2008} there exist constants $C,K>0$ and  a 
function $h:[K,\infty) \to (0,\infty)$ such that 
\begin{equation}\label{eq:fcth}
\E[Z \1_{\{ -h(x) \le Z \le x \}}]=0 \quad \mbox{and} \quad C^{-1} \le \frac{h(x)}{x} \le C
\end{equation}
for all $x\ge K$.\footnote{Here we assume that $p_+\,p_->0$. If either $p_+=0$ or $p_-=0$ one can proceed 
in a similar way by modifying $h$ slightly; we omit details.} We have 
\beao
I_2^{(n)}&=&
\sum_{i, j =1; i\neq j}^p   \sum_{t_1,t_2=1; t_1\neq t_2}^n Z_{i,t_1} Z_{j,t_1}Z_{i,t_2} Z_{j,t_2}\,
\big[\1_{A_{i,j,t_1,t_2}^c}+\1_{A_{i,j,t_1,t_2}}\big]=I_{21}^{(n)}+I_{22}^{(n)}\,,
\eeao
where $A_{i,j,t_1,t_2}=\{ -h(a_{np}^4) \le Z_{i,t_1} , Z_{j,t_1},Z_{i,t_2} ,Z_{j,t_2}\le a_{np}^4 \}$. We see that
\beao
\P(I_{21}^{(n)}>\epsilon\,a_{np}^4)&\le& (p\,n)^2\, \P(A_{i,j,t_1,t_2}^c)\le c\, (p\,n)^2\, \P(|Z| >\min(h(a_{np}^4), a_{np}^4))\\
&\le& c\, (pn)^2\, \P(|Z| > \min(C,C^{-1}) \,a_{np}^4)\le c\, (np)^{-2+\delta}\to 0,
\eeao
where we used the second formula in \eqref{eq:fcth}. The small constant $\delta>0$  comes from a Potter bound argument.
Finally, using the first condition in \eqref{eq:fcth}, we may  conclude similarly to \eqref{eq:sdvfg} that 
\begin{equation*}
P_{22}^{(n)}=\P(I_{22}^{(n)}>\epsilon\,a_{np}^4) \le c\, \frac{(pn)^2}{a_{np}^8}\, \big(\E[Z^2 \1_{\{ -h(a_{np}^4) \le Z \le a_{np}^4 \}}]\big)^4.
\end{equation*}
Since 
\begin{equation*}
\E[Z^2 \1_{\{ |Z| \le \max(C,C^{-1}) x \}}] \ge \E[Z^2 \1_{\{ -h(x) \le Z \le x \}}]\,,
\end{equation*}
and the \lhs\
is slowly varying (see \cite{feller}), we have $P_{22}^{(n)}\to 0$. The proof is complete for $\alpha\in (0,8/3)$. 
\end{proof}

\subsubsection*{The case $\alpha \in [8/3,4)$}
Before we can proceed with the case $\alpha\in [8/3,4)$ we provide an auxiliary result.
Consider the following decomposition 
\beao 
[\bfZ \bfZ' - \diag(\bfZ \bfZ')]^2=\bfD+\bfF+\bfR\,,
\eeao
where 
\beao
\bfD=(D_{ij})_{i,j=1,\ldots,p}=\diag([\bfZ \bfZ' - \diag(\bfZ \bfZ')]^2)\,,
\eeao
The $p\times p$ matrix $\bfF$ has a zero-diagonal and 
\begin{equation*}
F_{ij}= \sum_{u=1;u\neq i,j}^p \sum_{t=1}^n  Z_{it}\, Z_{jt}\, Z_{ut}^2 ,\quad 1\le i\ne j\le p\,,
\end{equation*}
The $p\times p$ matrix $\bfR$ has a zero-diagonal and 
\begin{equation*}
R_{ij}= \sum_{u=1;u\neq i,j}^p \sum_{t_1=1}^n \sum_{t_2=1; t_2 \neq t_1}^n Z_{i,t_1}\, Z_{j,t_2}\, Z_{u,t_1}\,Z_{u,t_2} ,\quad 
1 \le i\neq j\le p\,.
\end{equation*}
\begin{lemma}\label{lem:diagonal} 
Assume the conditions of Theorem~\ref{prop:offdiagonal} and $\alpha\in(2,4)$. 
Then $a_{np}^{-4} \big(\twonorm{\bfD}+\twonorm{\bfF}+\|\bfR\|_2\big)\stp 0$.
\end{lemma}
In view of this lemma we have 
\beao
a_{np}^{-4}\twonorm{\bfZ \bfZ' - \diag(\bfZ \bfZ')}^2= 
a_{np}^{-4}\twonorm{[\bfZ \bfZ' - \diag(\bfZ \bfZ')]^2}
=a_{np}^{-4}\twonorm{\bfD+\bfF+\bfR}^2\stp 0\,.
\eeao
This finishes the  proof of Theorem~\ref{prop:offdiagonal}. It is left to prove Lemma~\ref{lem:diagonal}.
\begin{proof}[Proof of the $\bfD$-part]
%Let $\alpha\in (2,4)$. An entry $D_{ii}$, $i \in \{ 1, \ldots, p\}$ of the matrix $D$ has the following form
We have for $i=1,\ldots,p$,
\begin{equation*}
\begin{split}
D_{ii}&= \sum_{u=1}^p \sum_{t=1}^n  Z_{it}^2 Z_{ut}^2 \1_{\{ i\neq u\}}
+\sum_{u=1}^p \sum_{t_1=1}^n \sum_{t_2=1}^n Z_{i,t_1} Z_{u,t_1}Z_{u,t_2} Z_{i,t_2} \1_{\{ i\neq u\}}\1_{\{ t_1\neq t_2\}} = M_{ii}+N_{ii}\,.
\end{split}
\end{equation*}
We write $\bfM$ and $\bfN$ for diagonal matrices constructed from $(M_{ii})$ and $(N_{ii})$ such that $\bfD=\bfM+\bfN$.
%the uncentered and centered parts of the diagonal of $\diag(\bfD)$. 
First bounding $\|\bfN\|_2$ by the Frobenius norm and then applying Markov's inequality, one can  prove that $a_{np}^{-4} \twonorm{\bfN}\stp 0$.
We have 
\begin{equation*}
\frac{\E[M_{ii}]}{a_{np}^{4}} \le c\, \frac{np}{a_{np}^{4}} \to 0, \qquad \nto.
\end{equation*}
Therefore centering of $M_{ii}$ will not influence the limit of the spectral norm $a_{np}^{-4}\|\bfM\|_2$.
Writing $A_{i,u}=\{|\sum_{t=1}^n  (Z_{it}^2 Z_{ut}^2 -\E[Z_1^2Z_2^2])\1_{\{ i\neq u\}}| > a_{np}^2 \}$,
we have for $i=1, \ldots,p$,
\beao
M_{ii}-\E[M_{ii}]&=& \sum_{u=1}^p \sum_{t=1}^n  \big(Z_{it}^2 Z_{ut}^2 -\E[Z_1^2Z_2^2]\big)\,\1_{\{ i\neq u\}}\,\big[\1_{A_{i,u}}+
\1_{A_{i,u}^c}\big]\\
%\1_{\{|\sum_{t=1}^n  (Z_{it}^2 Z_{ut}^2 -\E[Z_1^2Z_2^2])\1_{\{ i\neq u\}}| > a_{np}^2 \}}+
%& \quad + \sum_{u=1}^p 
%\Big(\sum_{t=1}^n  (Z_{it}^2 Z_{ut}^2 -\E[Z_1^2Z_2^2])
%\1_{\{|\sum_{t=1}^n  (Z_{it}^2 Z_{ut}^2 -\E[Z_1^2Z_2^2])\1_{\{ i\neq u\}}| \le a_{np}^2 \}}\Big]\\
&=& M_{ii}^{(1)} + M_{ii}^{(2)}. 
\eeao
On the one hand, $\twonorm{M^{(2)}}\le p\, a_{np}^2$. Hence $a_{np}^{-4} \twonorm{M^{(2)}}\stp 0$. On the other hand, 
we obtain with Markov's inequality, Proposition~\ref{prop:karasumsbeta} and the Potter bounds for $\epsilon>0$ and small $\delta>0$,
\beao
\P(\twonorm{M^{(1)}} >\epsilon\, a_{np}^{4}) &=& \P(\max_{i=1,\ldots, p} |M_{ii}^{(1)}| >\epsilon\,a_{np}^{4})\\
%&\le &(p\,n)\,\P(|\sum_{t=1}^n  (Z_{1t}^2 Z_{2t}^2 -\E[Z_1^2Z_2^2])| > a_{np}^2)
&\le& \P\Big(\max_{i=1,\ldots, p}  \sum_{u=1}^p \Big|\sum_{t=1}^n  (Z_{it}^2 Z_{ut}^2 -\E[Z_1^2Z_2^2])
\1_{\{ i\neq u\}}\,\1_{A_{i,u}}\Big|
%{\{ i\neq u\}}| > a_{np}^2 \}}  
>\epsilon\, a_{np}^{4}\Big)\\
&\le& c \frac{p^2}{a_{np}^{4}} \E\Big[ \Big|\sum_{t=1}^n  (Z_{1t}^2 Z_{2t}^2 -\E[Z_1^2Z_2^2])\Big| 
\1_{A_{1,2}}%\{|\sum_{t=1}^n  (Z_{1t}^2 Z_{2t}^2 -\E[Z_1^2Z_2^2])| > a_{np}^2 \}} 
\Big]\\
&\sim& c\, \frac{p^2}{a_{np}^{4}} \,n\, a_{np}^{2}\, \P(Z_1^2Z_2^2>a_{np}^2) \le\dfrac{ p \,(np)^\delta}{a_{np}^{2}} \to 0,
\eeao
since $Z_1Z_2$ is \regvary\ with index $\alpha$. This finishes the proof of the $\bfD$-part.
\end{proof}
%\begin{lemma}\label{lem:Foffdiagonal}
%Assume $\alpha\in(2,4)$.
%Then $a_{np}^{-4} \twonorm{F}$ converges to $0$ in probability.
%\end{lemma}
\begin{proof}[Proof of the $\bfF$-part]
%The idea is to split $F$ into several parts.
Let $\delta>0$. We will use the following decomposition for $i\ne j$:
\begin{equation*}
F_{ij}= \sum_{u=1;u\neq i,j}^p \sum_{t=1}^n  Z_{it} Z_{jt} (Z_{ut}^2 - \E[Z^2 \1_{\{Z^2 \le a_{np}^{4-2\delta}\}}])+ \E [Z^2\1_{\{Z^2 \le a_{np}^{4-2\delta}\}}]\,(p-2) 
\sum_{t=1}^n  Z_{it} Z_{jt} =\wt F_{ij}+T_{ij}\,.
\end{equation*}
We observe that $\bfT=  \E [Z^2\1_{\{Z^2 \le a_{np}^{4-2\delta}\}}] \,(p-2) \,(\bfZ_n \bfZ_n' - \diag(\bfZ_n \bfZ_n'))$.
We have for some constant $c>0$, 
\beao
\|\bfT\|_2^2&=&\|\bfT^2\|_2\le
c\,p^2\,
\|(\bfZ_n \bfZ_n' - \diag(\bfZ_n \bfZ_n'))^2\|_2\\ %&\le& c\,p^2\,\big[\|\bfD+\bfF+\bfR\|_2\big]\\
&\le &c\,p^2\,\|\bfD+\wt \bfF+\bfR\|_2+ c\,p^2 \,\|\bfT\|_2\,.
\eeao
Therefore
\beam\label{eq:inequality}
\dfrac{\|\bfT\|_2}{a_{np}^4}\le c\,\dfrac{p}{a_{np}^2}\,
\Big(\dfrac{\|\bfD+\wt \bfF+\bfR\|_2}{a_{np}^4}\Big)^{1/2}+
c\,\dfrac{p}{a_{np}^2} \Big(\dfrac{\|\bfT\|_2}{a_{np}^4}\Big)^{1/2}\,. 
\eeam
In the course of the proof of this lemma we show that
\beao
\dfrac{\|\bfD+\wt \bfF+\bfR\|_2}{a_{np}^4}\stp 0\,.
\eeao
Moreover, there is a small $\vep>0$ \st
\begin{equation*}
\delta_n=\frac{p}{a_{np}^{2}}\le n^{1-4/\alpha+\vep}, \quad 1-4/\alpha+\vep<0\,.
\end{equation*}
Therefore iteration of \eqref{eq:inequality} yields for $k\ge 1$
\beam\label{eq:ko}
\dfrac{\|\bfT\|_2}{a_{np}^4}&\le& o_\P(1)
+ c\,\delta_n\Big(\delta_n\,
\Big(\dfrac{\|\bfD+\wt \bfF+\bfR\|_2}{a_{np}^4}\Big)^{1/2}\Big)^{1/2}
+c\,\delta_n \Big(\delta_n \Big(\dfrac{\|\bfT\|_2}{a_{np}^4}\Big)^{1/2}\Big)^{1/2}\nonumber\\
&=& o_\P(1)+c\,\Big(\delta_n^{4+2}\dfrac{\|\bfT\|_2}{a_{np}^4}\Big)^{1/4}\nonumber\\
&\le & o_\P(1)+c\,\Big(\delta_n^{2^k+\cdots +2}\dfrac{\|\bfT\|_2}{a_{np}^4}\Big)^{1/{2^k}}\,.
\eeam
Using some elementary moment bounds for $\|\bfT\|_2$ (e.g. a bound by the Frobenius norm),
it is not difficult to show that $n^{-l} \|\bfT\|_2\stp 0$ for some sufficiently large $l$. Thus we achieve that the \rhs\
in \eqref{eq:ko} converges to zero in \pro y. 
%The problem of showing that $a_{np}^{-4} \twonorm{T} \cip 0$ is almost the same as in Proposition~\ref{prop:offdiagonal} itself, where one was faced with $a_{np}^{-2} \twonorm{\bfZ_n \bfZ_n' - \diag(\bfZ_n \bfZ_n')}$. Fortunately, it turns out that the factor 
\par
It remains to show that $a_{np}^{-4} \twonorm{\widetilde{\bfF}}\stp 0$.
With the notation 
$B_{u,t}=\{Z_{ut}^2 \le a_{np}^{4-2\delta}  \} $
for some small $\delta>0$,
we decompose $Z_{it} Z_{jt} (Z_{ut}^2 - \E[Z^2 \1_{\{Z^2 \le a_{np}^{4-2\delta}\}}])$ as follows:
\begin{equation*}
\begin{split}
&Z_{it} Z_{jt} (Z_{ut}^2 \1_{B_{u,t}} - \E[Z^2 \1_{\{Z^2 \le a_{np}^{4-2\delta}\}}]) + Z_{it} Z_{jt} Z_{ut}^2 \1_{B_{u,t}^c}\,.
\end{split}
\end{equation*}
We decompose the matrix $\widetilde{\bfF}$ accordingly: 
\begin{equation*}
\widetilde{\bfF}= \widetilde{\bfF}^{(1)} + \widetilde{\bfF}^{(2)}\,,
\end{equation*}
\st, for example, 
\begin{equation*}
\widetilde{F}^{(1)}_{ij} = \sum_{u=1;u\neq i,j}^p \sum_{t=1}^n  Z_{it} Z_{jt} (Z_{ut}^2 \1_{B_{u,t}} - \E[Z^2 \1_{\{Z^2 \le a_{np}^{4-2\delta}\}}])\,, \quad i \neq j.
\end{equation*} 
%In what follows we are going to prove that $a_{np}^{-4} \twonorm{F^{(kl)}}, k,l \in \{1,2\}$ converge to $0$ in probability.
\noindent
$\widetilde{\bfF}^{(1)}$: 
Bounding the spectral norm by the Frobenius norm, applying Markov's inequality and using Karamata's theorem together with the Potter bounds one can check that for $\epsilon>0$ and small $\delta>0$,
\begin{equation*}
\begin{split}
\P&(\twonorm{\widetilde{\bfF}^{(1)}}> \epsilon\, a_{np}^4) \le c\, a_{np}^{-8}\, \E \Big[ \sum_{i,j=1}^p (\widetilde{F}^{(1)}_{ij})^2  \Big]\\
&\le  c \,\frac{p^3\, n}{a_{np}^{8}}\E[(Z_1Z_2)^2 ] \E[(Z^2 \1_{\{Z^2 \le a_{np}^{4-2\delta}\}} - \E[Z^2 \1_{\{Z^2 \le a_{np}^{4-2\delta}\}}])^2 ]\\
&\le c \,\frac{p^3\, n}{a_{np}^{8}} \E[Z^4 \1_{\{Z^2 \le a_{np}^{4-2\delta}\}}] \\
&\le c \frac{p^3 n}{a_{np}^{4 \delta}} \P(|Z|> a_{np}^{2-\delta}) \to 0\,,
%\le c \frac{p}{n} (np)^{-\delta(4/\alpha -1)+ \vep} \to 0, 
\qquad \nto,
\end{split}
\end{equation*}
%since $-\delta(4/\alpha -1)+ \vep<0$ for $\vep$ small enough.
$\widetilde{\bfF}^{(2)}$: %The set $B_{u,t}^c$ defined in \eqref{eq:ab}  is contained in 
%\beao
%\{ |Z_{ut}^2-\E[Z^2]|>\min(C,C^{-1})a_{np}^{4-2\delta}  \}.
%\eeao
%Since $\twonorm{\cdot} \le \inftynorm{\cdot}$ for symmetric matrices, we get
We have for small $\delta>0$,
\beao%\label{eq:boundf12}
\P(\twonorm{\widetilde{\bfF}^{(2)}}> \epsilon\, a_{np}^4) &\le& 
\P\Big(\bigcup_{1\le u\le p\,,1\le t\le n} B_{u,t}^c\Big)\\
&\le & p\,n\, \P(|Z|> a_{np}^{2-\delta})\to 0\,,\qquad\,\nto\,.
\eeao
The proof of the $\bfF$-part is complete.
\end{proof}
\begin{proof}[Proof of the $\bfR$-part]

We have
\beao
\E [\|\bfR\|_2^2]\le \E[\|\bfR\|_F^2]&\le& \sum_{i,j=1}^p \sum_{u=1}^p \sum_{t_1=1}^n \sum_{t_2=1}^n (\E[Z^2])^4\le c\, p^3\, n^2\,.
\eeao
Therefore and by Markov's inequality for $\epsilon>0$,
\begin{equation}\label{eq:ljo}
\P(\twonorm{\bfR}> \epsilon\, a_{np}^4) %\le \P(\|R\|_F^2> c a_{np}^8)
%\le c \,\frac{\E[\|R\|_F^2]}{a_{np}^8}
\le c\, \frac{p^3 \,n^2}{a_{np}^8}
 \to 0, \qquad \nto\,,
\end{equation}
as long as $\alpha \in (2,16/5)$.
For $\alpha \in [16/5,4)$  
we use a similar idea for the truncated entries. Write $\bfR=\overline{\bfR} + \widetilde{\bfR}$, where for $i\neq j$
\begin{equation*}
\begin{split}
\overline{R}_{ij} &= \sum_{u=1;u\neq i,j}^p \sum_{t_1=1}^n \sum_{t_1=1; t_1 \neq t_2}^n Z_{i,t_1} Z_{j,t_2} Z_{u,t_1}Z_{u,t_2}\, 
\1_{A_{i,j,t_1,t_2}},\\
\widetilde{R}_{ij} &= \sum_{u=1;u\neq i,j}^p \sum_{t_1=1}^n \sum_{t_1=1; t_1 \neq t_2}^n Z_{i,t_1} Z_{j,t_2} Z_{u,t_1}Z_{u,t_2}\, 
\1_{A_{i,j,t_1,t_2}^c},
\end{split}
\end{equation*}
with $A_{i,j,t_1,t_2}^c=\{ -h(a_{np}) \le Z_{i,t_1}, Z_{j,t_2}, Z_{u,t_1},Z_{u,t_2}\le a_{np} \}$ and $h$ as in \eqref{eq:fcth}. 
Analogously to \eqref{eq:ljo}, using the fact that the \rv s   $Z_{i,t_1} Z_{j,t_2} Z_{u,t_1}Z_{u,t_2}\, 
\1_{A_{i,j,t_1,t_2}}$ are uncorrelated for the considered index set,
one obtains for $\epsilon>0$, 
\beao
\P(\twonorm{\overline{\bfR}}> \epsilon\, a_{np}^4) &\le& c \,\frac{p^3\, n^2}{a_{np}^8}
 \,\E[Z^2 \1_{\{ |Z| > \min(C,C^{-1})\,a_{np}\}}]\\ &\le& c \,\frac{p^3 \,n^2}{a_{np}^6} \,\P(|Z| > \min(C,C^{-1}) \,a_{np}) 
%\le c\, \frac{p^2\, n}{a_{np}^6} 
\to 0
\eeao
as $\nto$, where we used Karamata's theorem and $\P(A_{i,j,t_1,t_2}) \le c\, \P(|Z|>\min(C,C^{-1})\,a_{np})$.
\par
%In the following we apply the Frobenius norm to higher powers of $\widetilde{R}$. 
We introduce the truncated \rv s $\wt Z_{it}=Z_{it}\1_{\{ -h(a_{np}) \le Z_{it} \le a_{np} \}}$ with generic element $\wt Z$. 
We will repeatedly use the following inequality which is valid
for a real symmetric matrix $\bfM$:
\beao
\twonorm{\bfM}^2 \le \frobnorm{\bfM}^2 = \tr(\bfM^2).
\eeao
Then we have for $k\ge 1$, $\twonorm{\widetilde{\bfR}^{2^{k-1}}}^2 = \twonorm{\widetilde{\bfR}^{2^k}}$ and 
\beao
 \twonorm{\widetilde{\bfR}}^{2^k} \le \tr(\widetilde{\bfR}^{2^k})= \sum_{i,j=1}^p (\widetilde{R}^{2^{k-1}})_{ij}^2\,.
\eeao
This together with the Markov inequality of order $2^k$ yields
\begin{equation}\label{eq:sdfgb}
\P(  \twonorm{\widetilde{\bfR}}> ca_{np}^{4}) \le c a_{np}^{-4 \cdot 2^k} \E \Big[\sum_{i,j=1}^p (\widetilde{R}^{2^{k-1}})_{ij}^2 \Big] .
\end{equation}
Next we study the structure of $\widetilde{\bfR}^{2^{k-1}}$. The $(i,j)$-entry of this matrix is
\begin{equation}\label{eq:powerexpl}
(\widetilde{R}^{2^{k-1}})_{ij} = \sum_{i_1=1}^p %\sum_{i_2=1}^p  
\cdots \sum_{i_{2^{k-1}-1}=1}^p \widetilde{R}_{i,i_1} \widetilde{R}_{i_1,i_2}\cdots  \widetilde{R}_{i_{2^{k-1}-2},i_{2^{k-1}-1}} \widetilde{R}_{i_{2^{k-1}-1},j}.
\end{equation}
In view of \eqref{eq:powerexpl} and by definition of $\widetilde{\bfR}$, $(\widetilde{R}^{2^{k-1}} )_{ij}$ contains
exactly $2^k-1$ sums running from 1 to $p$, and $2^k$ sums running from 1 to $n$. 
Now we consider the expectation on the right-hand side of \eqref{eq:sdfgb}. 
The highest and lowest powers of $\wt Z_{it}$ in this expectation are $2^k$ and $1$. Let $(I,T) = ((i_1,t_1), \ldots,(i_{2^k},t_{2^k}))$. We have 
\begin{equation*}
\E \Big[\sum_{i,j=1}^p (\widetilde{R}^{2^{k-1}})_{ij}^2 \Big] = \sum_{(I,T)\in S} \E[\widetilde{Z}_{i_1,t_1} \widetilde{Z}_{i_2,t_2} \cdots \widetilde{Z}_{i_{2^k},t_{2^k}} 
],
\end{equation*}
where $S\subset \{1, \ldots,p\}^{2^k} \times \{1, \ldots,n\}^{2^k}$ is the index set that 
covers all combinations of indices that arise on the left-hand side. 
Since $\E[\widetilde{Z}]=0$, each $\widetilde{Z}$ in $\widetilde{Z}_{i_1,t_1} \widetilde{Z}_{i_2,t_2} \cdots \widetilde{Z}_{i_{2^k},t_{2^k}} $ 
must appear at least twice for the expectation of this product to be non-zero. Let $S_1 \subset S$ be the set of all those indices 
that make  a non-zero contribution to the sum. From the specific structure of $\widetilde{\bfR}$,  \eqref{eq:powerexpl} and the considerations above it now follows that the cardinality of $S_1$ has the following bound
\begin{equation*}
|S_1| \le c(k)\, p^2\, p^{2^k-1}\, n^{2^k}=c\, p^{2^k+1}\,n^{2^k}\,.
\end{equation*}
For $l=2,3$ we can use $\E[|\widetilde{Z}^l|]\le c$. If $l\ge 4$, we infer with Karamata's theorem 
\begin{equation}\label{eq:23424}
\E[|\widetilde{Z}^l|] \le c\, a_{np}^{l}\, \P(|Z| > a_{np}).
\end{equation}
The subset of $S_1$ (say $S_l$) which generates a  $\widetilde{Z}^l$ for $l\ge 4$ is much smaller than $S_1$. Also its cardinality is divided by at least $n$ if we go from $l$ to $l+1$, i.e. $|S_l|\ge n |S_{l+1}|$. Observe that $n a_{np}^{-1}$ converges to infinity. This combined with \eqref{eq:23424} tells us that only the case of every $\widetilde{Z}$ appearing exactly twice is of interest since it has most 
influence on the expectation in \eqref{eq:sdfgb}. We conclude that
\begin{equation*}
\begin{split}
\dfrac{1}{a_{np}^{4 \cdot 2^k}}\, \E \Big[\sum_{i,j=1}^p (\widetilde{R}^{2^{k-1}})_{ij}^2 \Big]
&\le c\, \dfrac{|S_1|}{a_{np}^{4 \cdot 2^k}}\le c\,p\, \Big( \frac{np}{a_{np}^4} \Big)^{2^k}\le c\, (np)\,  \Big( \frac{np}{a_{np}^4} \Big)^{2^{k}}.
\end{split}
\end{equation*}
The  expression on the \rhs\ converges to $0$ if $1+2^{k}-2^{k+2}/\alpha <0$ or equivalently
\beao
k>\log\Big( \frac{\alpha}{4-\alpha}\Big) (\log 2)^{-1}.
\eeao
Since $k$ was arbitrary the proof of the $\bfR$-part is finished.
\end{proof}

\subsection{Proof of \eqref{eq:eigena}}
We define the $p\times p$ matrix $\Y_n^\rightarrow$ as the diagonal matrix with elements
\begin{equation*}
(\Y_n^\rightarrow)_{ii}= \max_{t=1, \ldots,n} Z_{it}^2\,, \qquad i=1,\ldots,p\,.
\end{equation*}
Correspondingly, we define the $n\times n$ matrix $\Y_n^\downarrow$ as the diagonal matrix with elements
\begin{equation*}
(\Y_n^\downarrow)_{tt}= \max_{i=1, \ldots,p} Z_{it}^2\,, \qquad t=1,\ldots,n\,.
\end{equation*}
\begin{lemma}\label{cor:688}
Assume the conditions of Theorem~\ref{thm:iidmain}. 
\begin{enumerate}
\item
If $\beta\in ((\alpha/2-1)_+,1]$ we have
%$\beta\in (0,1]\cap (\alpha/2-1,1]$ we have
\beao
a_{np}^{-2}\,\max_{i=1,\ldots,p}\big|\la_{(i)}-\la_{(i)}(\Y_n^\rightarrow)\big|\stp 0\,,\qquad\nto \,.
\eeao
\item 
If $\beta^{-1}\in ((\alpha/2-1)_+,1)$ we have
%$\beta\in (0,1]\cap (\alpha/2-1,1]$ we have
\beao
a_{np}^{-2}\,\max_{i=1,\ldots,n}\big|\la_{(i)}-\la_{(i)}(\Y_n^\downarrow)\big|\stp 0\,,\qquad\nto \,.
\eeao
\end{enumerate}
\end{lemma}
\begin{proof}
We restrict ourselves to the proof in the case $\beta\in (0,1]$; the case $\beta>1$ can again be handled by switching 
from $\bfZ\bfZ'$ to $\bfZ'\bfZ$.
An application of Weyl's inequality (see \eqref{eq:weylin}) and the triangle inequality yield
\begin{equation*}
\begin{split}
a_{np}^{-2}\,\max_{i=1,\ldots,p}\big|\la_{(i)}-\la_{(i)}(\Y_n^\rightarrow)\big|& \le 
a_{np}^{-2} \twonorm{\bfZ \bfZ' - \diag(\bfZ \bfZ')} + a_{np}^{-2} \twonorm{ \diag(\bfZ \bfZ')- \diag(\Y_n^\rightarrow)}.
\end{split}
\end{equation*}
The first term on the right-hand side converges to $0$ in probability by Theorem~\ref{prop:offdiagonal}(1).
As regards the second term we have 
\beao
a_{np}^{-2} \twonorm{ \diag(\bfZ \bfZ')- \diag(\Y_n^\rightarrow)}&=& a_{np}^{-2}\max_{i=1,\ldots,p} \Big|D_i^\rightarrow - \max_{t=1,\ldots,n} Z_{it}^2\Big|\,.
\eeao
The \rhs\ converges to zero in \pro y in view of Lemma~\ref{thm:keyresult} applied to $(Z_{it}^2)$.
\end{proof}
Now \eqref{eq:eigena} follows from the next result.
\begin{lemma}\label{lem:redorder}
Assume the conditions of Theorem~\ref{thm:iidmain}. 
\begin{enumerate}
\item
If $\beta\in ((\alpha/2-1)_+,1]$ we have
%$\beta\in (0,1]\cap (\alpha/2-1,1]$ we have
\beao
a_{np}^{-2}\,\max_{i=1,\ldots,p}\big|\la_{(i)}(\Y_n^\rightarrow)-Z_{(i),np}^2\big|\stp 0\,,\qquad\nto \,.
\eeao
\item 
If $\beta^{-1}\in ((\alpha/2-1)_+,1)$ we have
%$\beta\in (0,1]\cap (\alpha/2-1,1]$ we have
\beao
a_{np}^{-2}\,\max_{i=1,\ldots,n}\big|\la_{(i)}(\Y_n^\downarrow)-Z_{(i),np}^2\big|\stp 0\,,\qquad\nto \,.
\eeao
\end{enumerate}
\end{lemma}
%\begin{theorem}\label{thm:main}
%Assume the conditions of Theorem~\ref{thm:iidmain}.
%$\alpha \in (0,4)$. { Check $\beta^{-1}$} If $\beta \in (0,1] \cap (\alpha/2-1,1]$ then 
%\begin{equation*}
%a_{np}^{-2} \max_{i=1,\ldots,p} |\lambda_{(i)}-Z_{(i),np}^2| \cip 0\,, \quad \nto.
%\end{equation*}
%\end{theorem}
\begin{proof}
We focus on part (1). We write 
$V_{(1)}\ge \cdots \ge V_{(p)}$ for the order statistics of $(\max_{t=1,\ldots,n}Z_{it}^2)$. By definition of the order statistics we have $Z_{(i),np}^2\ge V_{(i)}$ for $i=1,\ldots,p$. We choose $\delta$ such that $1>\delta >\frac{2+\beta}{2(1+\beta)}$ and define the event
\beao
B_{np}^{2\delta}= \{ \mbox{There is a row of $(Z_{it}^2)$ with at least two entries larger than $a_{np}^{2\delta}$.} \}\,.
\eeao
By Lemma~\ref{lem:onecolumn}, $\P(B_{np}^{2\delta})\to 0$. 

Next, we choose $0<\vep < 1-\delta$. Then Lemma~\ref{prop:seqk} guarantees the existence of a sequence $k=k_n \to \infty$ such that the event 
\begin{equation*}
\Omega_n = \{ Z_{(k),np}^2 > a_{np}^{2(1-\vep)} \}
\end{equation*}
satisfies $\P(\Omega_n^c)\to 0$. On the event $(B_{np}^{2\delta})^c \cap \Omega_n$ we have 
\begin{equation*}
V_{(i)}-Z_{(i),np}^2=0\,, \quad i=1,\ldots,k.
\end{equation*}
This shows for $\gamma>0$,
\begin{equation*}
\begin{split}
\limsup_{\nto}\P\big(a_{np}^{-2}\max_{i=1,\ldots,p} |V_{(i)}-Z_{(i),np}^2|>\gamma\big)
&\le \limsup_{\nto}\P\big(\{a_{np}^{-2}\max_{i=1,\ldots,p} |V_{(i)}-Z_{(i),np}^2|>\gamma\}\cap (B_{np}^{2\delta})^c \cap \Omega_n\big)\\
&\quad + \limsup_{\nto} \P(B_{np}^{2\delta})  +\limsup_{\nto} \P(\Omega_n^c)\\
&= \limsup_{\nto}\P\big(\{a_{np}^{-2}\max_{i=k+1,\ldots,p} |V_{(i)}-Z_{(i),np}^2|>\gamma\}\cap (B_{np}^{2\delta})^c \cap \Omega_n\big)\\
&\le \limsup_{\nto}\P\big( 2 \, a_{np}^{-2} \,Z_{(k+1),np}^2 >\gamma \big) =0.
\end{split}
\end{equation*}
\end{proof}

%--------------------------------------------------------------
\section{Generalization to autocovariance matrices}\label{sec:autocov}\setcounter{equation}{0}

An important topic in multivariate time series analysis is the study of the covariance structure. From the field $(Z_{it})$ we construct the $p\times n$ matrices
\beao
\z(s,k)=\z_n(s,k)=(Z_{i-s,t-k})_{i=1,\ldots,p;t=1,\ldots,n}\,, \quad s,k\in\Z\,.
\eeao
We introduce the (non-normalized) {\em generalized sample autocovariance matrices}  
\beao%\label{eq:Ygsrg}
(\z(0,0)\z(s,k)')\,, \quad s,k\in\Z\,,
\eeao
with entries
\beao%\label{eq:Ygsrg}
(\z(0,0)\z(s,k)')_{ij} = \sum_{t=1}^n Z_{i,t} \,Z_{j-s,t-k}\,,\qquad i,j=1,\ldots,p\,.
\eeao
If $\min(|s|,|k|) \neq 0$, the generalized sample autocovariance matrix $\z(0,0)\z(s,k)'$ is not symmetric and might thus have complex eigenvalues. In what follows, we will be interested in the {\em singular values} $\la_1(s,k), \ldots, \la_p(s,k)$ of $\z(0,0)\z(s,k)'$.  The singular values of a matrix $\bfA$ are the square roots of the eigenvalues of $\bfA\bfA'$. We reuse the notation $(\la_i(s,k))$ for the singular values 
%since in the case $s=k=0$ the singular values of the sample covariance matrix coincide with its eigenvalues.
and again write $\la_{(1)}(s,k)\ge \cdots  \ge\la_{(p)}(s,k)$ for their order statistics.
\par
%{ This notation has been introduced before.}
%We need some more notation. For $s\in \Z$ let $s_+, s_-\ge 0$ be the positive and negative parts of $s$, respectively. 
%We will use the order statistics 
%\begin{equation*}
%Z_{(1),(n(p-|s|)}^2 \ge Z_{(2),(n(p-|s|)}^2 \ge  \cdots \ge Z_{(n(p-|s|)),(n(p-|s|)}^2, \qquad n,p\ge 1\,,
%\end{equation*}
%of the random variables $(Z_{it}^2)$, $i=1+s_-,\ldots,p-s_+; t=1,\ldots,n$.
%Important objects are the row-sums 
%\begin{equation*}%\label{eq:ll}
%D_i^{\rightarrow}=D_i^{(s),(n),\rightarrow}=\sum_{t=1}^n Z_{it}^2\,, \qquad
%i=1+s_-,\ldots,p-s_+;\quad  s\in\Z, n=1,2,\ldots\,,
%\end{equation*} 
%and their ordered values
%$D_{(1)}^{(s),\rightarrow}\ge \cdots \ge D_{(p-|s|)}^{(s),\rightarrow}$. Similarly we also define the order statistics $D_{(1)}^{(s),\downarrow}\ge \cdots \ge D_{(n-|s|)}^{(s),\downarrow}$ of the column-sums $D_t^{\downarrow}$.
\begin{theorem}\label{thm:iidauto}
Assume $s,k\in \Z$. Consider the $p\times n$-dimensional matrices $\z(0,0)$ and $\z(s,k)$ with iid entries. We assume the following conditions:
\begin{itemize} \item
The \regvar\ condition \eqref{eq:regvar} for some  
$\alpha \in (0,4)$. 
\item $\E [Z]=0$ for $\alpha\ge 2$.
\item
The integer \seq\ $(p_n)$ 
has growth rate \ref{eq:p} for some  $\beta\ge 0$.
\end{itemize}
\begin{itemize}
\item[(1)] If $k\neq 0$, then 
\beao
a_{np}^{-2} \, \la_{(1)}(s,k) \cip 0\,.
\eeao
\end{itemize}
Now assume $k=0$ and recall the notation $D_{(i)}^{\rightarrow}$ and $D_{(i)}^{\downarrow}$ from Section \ref{sec:section2.2}. Then the following statements hold:
\begin{itemize}
\item[(2)] If $\beta\in [0,1]$, then
\beam\label{eq:rowauto}
a_{np}^{-2}\,\max_{i=1,\ldots,p-|s|}\big|\la_{(i)}(s, 0)-D_{(i)}^{\rightarrow}\big|\stp 0\,.
\eeam

\item[(3)] If $\beta > 1$, then
\beam\label{eq:colauto}
a_{np}^{-2}\,\max_{i=1,\ldots,n-|s|}\big|\la_{(i)}(s, 0)-D_{(i)}^{\downarrow}\big|\stp 0\,.
\eeam

\item[(4)] If  $\min(\beta,\beta^{-1} ) \in ((\alpha/2-1)_+,1]$, then
\beam\label{eq:eigenauto}
a_{np}^{-2}\,\max_{i=1,\ldots,p-|s|}\big|\la_{(i)}(s, 0)-%Z_{(i),(n(p-|s|)}^2 
Z_{(i),np}^2\big|\stp 0\,.
\eeam
\end{itemize}
\end{theorem}
\begin{proof}
We focus on the case $\beta \in [0,1]$.  The proof is analogous to the proof of
Theorem~\ref{thm:iidmain} which was given in Section~\ref{sec:proofs}. This proof relied on the reduction of $\z\z'$ to its diagonal. If $k=0$, we will reduce $\z(0,0)\z(s,k)'$ to 
 a $p\times p$ matrix $\M^{(s,k)}$, which only takes values on its $s$th sub-diagonal.
The entries of the $s$th sub-diagonal of $\M^{(s,k)}$ are $\M^{(s,k)}_{i,i+s}$, $i=1+s_-,\ldots,p-s_+$. Here $s_+, s_-\ge 0$ are the positive and negative parts of $s$, respectively.

We sketch the steps of this reduction. Let $k\in\Z$. For simplicity of notation assume $s \ge 0$.
Define the $p\times p$ matrix $\M^{(s,k)}$,
\begin{equation*}
\M^{(s,k)}_{i,i+s}=\1_{\{ k=0\}}(\z(0,0)\z(s,0)')_{i,i+s}= \1_{\{ k=0\}}\sum_{t=1}^n Z_{it}^2\,, \quad i=1,\ldots,p-s\,,
\end{equation*}
and $\M^{(s,k)}_{ij}=0$ for all other $i,j$. We have
\begin{equation*}
\begin{split}
\big((\z(0,0)\z(s,k)'-&\M^{(s,k)}) (\z(0,0)\z(s,k)'-\M^{(s,k)})'\big)_{ij}\\&=
\sum_{u=1}^p \sum_{t_1=1}^n \sum_{t_2=1}^n Z_{i,t_1} Z_{j,t_2} Z_{u-s,t_1-k}Z_{u-s,t_2-k} \1_{\{ i\neq u-s, j\neq u-s \}}\\
&\quad \times (\1_{\{ i=j\}}+\1_{\{i\neq j, t_1=t_2\}}+\1_{\{i\neq j, t_1\neq t_2\}})\\
&= \bfD_{ij}+\bfF_{ij}+\bfR_{ij}\,.
\end{split}
\end{equation*}
Repeating the steps in the proof of Lemma~\ref{lem:diagonal}, one obtains 
\beao
a_{np}^{-4}\twonorm{\bfD+\bfF+\bfR}^2\stp 0\,.
\eeao
Therefore we also have
\begin{equation*}
\begin{split}
a_{np}^{-4}\twonorm{\z(0,0)\z(s,k)'-\M^{(s,k)}}^2 &= 
a_{np}^{-4}\twonorm{(\z(0,0)\z(s,k)'-\M^{(s,k)}) (\z(0,0)\z(s,k)'-\M^{(s,k)})'}\stp 0\,.
\end{split}
\end{equation*}

This proves part (1). Since, with probability tending to $1$, the matrix $\M^{(s,k)}$ has the required singular values, part (2) follows by Weyl's inequality.  

Finally, part (4) is a consequence of Lemma~\ref{lem:redorder}.
\end{proof}
We obtain the following result for the weak \con\ of the point processes of the points $\la_i(s,0)$, $s=0,\ldots,l$; the proof is similar
to the one of Theorem~\ref{thm:iidauto}. 
\begin{corollary}
Assume the conditions of Theorem~\ref{thm:iidauto}. 
%Let $\xi_{\alpha/2}$ be a \Frechet distributed random variable with parameter $\alpha/2$.
Then, with the notation of Theorem~\ref{cor:1}, the following \pp\ \con\ holds for $ l\ge 0$ and $(\beta,\beta^{-1})\in ((\alpha/2-1)_+,1]$,
%\begin{enumerate}
%\item If either $\alpha\in (0,2)$ and $\beta\ge 0$, or $\alpha\in [2,4)$ and $\min(\beta,\beta^{-1} ) \in ((\alpha/2-1)_+,1]$, then
\beao
\sum_{i=1}^{ p} \vep_{a_{np}^{-2} \, \big(\la_{ (i)}({ 0},0),\ldots,\la_{ (i)}(l,0)\big)} \cid \sum_{i=1}^\infty 
\vep_{\Gamma_i^{-2/\alpha}\big(1,\ldots,1\big)} \,.
\eeao
%and,
%{\blue delete, not used in paper with Richard} for $\gamma<\alpha/2$,
%the \seq\ $\big((a_{np}^{-2} \la_{(1)}(s,0))^\gamma\big)_{n\ge 1}$ is uniformly integrable.
%\item If  $\alpha\in (2,4)$ and $\beta\ge 0$, then
%\begin{equation*}
%a_{np}^{-2} \, (\la_{(1)}(s,0) -(n\vee p) \E[Z^2] ) \cid \xi_{\alpha/2}\,,
%\end{equation*}
%{\blue delete this part, not used in the paper with Richard} and, for $\gamma<\alpha/2$,
%the \seq\ $\big((a_{np}^{-2} |\la_{(1)}(s,0) -(n\vee p) \E[Z^2] |)^\gamma\big)_{n\ge 1}$ is uniformly integrable.
%\end{enumerate}
\end{corollary}
The joint \con\ of a finite number of the random variables $\la_{(i)}(s,0)$, $i\ge 1$, $s\ge 0$, is an immediate con\seq\ of this result.
%{\red 
%\begin{proof}
%The distributional convergence in part (1) follows from Theorem~\ref{thm:iidauto}. Indeed, if $\alpha\in (0,2)$ and $\beta\ge 0$
%we have $a_{np}^{-2} D_{(1)}^{\rightarrow}\cid \xi_{\alpha/2}$ and $a_{np}^{-2} D_{(1)}^{\downarrow}\cid \xi_{\alpha/2}$ and if 
%$\alpha\in (2,4)$ and $\min(\beta,\beta^{-1} ) \in ((\alpha/2-1)_+,1]$ we have the corresponding results $a_{np}^{-2} (D_{(1)}^{\rightarrow}-n\,\E[Z^2])\cid \xi_{\alpha/2}$
%and $a_{np}^{-2} (D_{(1)}^{\downarrow}-p\,\E[Z^2])\cid \xi_{\alpha/2}$; if $\alpha=1$ one needs the corresponding modification for the form of the centering constants. However, we also have $(p\vee n)/a_{np}^2\to 0$ \st\ centering is not needed.
%\begin{equation*}
%a_{np}^{-2} Z_{(1),(n(p-|s|)}^2 \cid \xi_{\alpha/2}\,,
%\end{equation*}
%if $\alpha\in [2,4)$ and $\min(\beta,\beta^{-1} ) \in ((\alpha/2-1)_+,1]$.
%\par
%The same arguments combined with Lemma~\ref{lem:pp} conclude the proof of part (2).
%\end{proof}

%-----------------------------------------------------------------------
\appendix

\section{Regular variation, large deviations and point processes}\label{appendix:A}\setcounter{equation}{0}

Let $(Z_i)$ be iid copies of $Z$ whose distribution satisfies 
\begin{equation*}
\P(Z>x)\sim p_+ \dfrac{L(x)}{x^{\alpha}}\quad\mbox{and}\quad  \P(Z\le -x)\sim p_-
\dfrac{L(x)}{x^{\alpha}} \,,\quad \xto\,,
\end{equation*}
 for some tail index $\alpha>0$,
where $p_+,p_-\ge 0$ with $p_++p_-=1$ and $L$ is a slowly varying function. We say that $Z$ is regularly varying with index $\alpha$. The monograph \cite{bingham:goldie:teugels:1987} contains many properties and useful tools for regularly varying functions. Theorem 1.5.6 therein, which is known as Potter bounds, asserts that a regularly varying function essentially lies between two power laws. In particular,  
for any $\delta>0$ and $C>1$ we have for $x$ sufficiently large,
\begin{equation*}
C^{-1} x^{-\delta} \le L(x) \le C x^{\delta}\,.
\end{equation*}

Theorem 1.6.1 in \cite{bingham:goldie:teugels:1987}, widely known as Karamata's theorem, describes the behavior of truncated moments of the regularly varying random variable $Z$. For $x\to \infty$,
\begin{equation*}
\begin{split}
\E[|Z|^\beta \1_{\{|Z|\leq x\}}] &\sim \frac{\alpha}{\beta -\alpha} x^\beta \P(|Z|>x), \quad \beta>\alpha,\\
\E[|Z|^\beta \1_{\{|Z| > x\}}] &\sim \frac{\alpha}{\alpha-\beta} x^\beta \P(|Z|>x), \quad \beta<\alpha.
\end{split}
\end{equation*}

If $\E[|Z|]<\infty$ also assume $\E[Z]=0$. The product $Z_1Z_2$ is regular varying with the same index $\alpha$ and $\P(|Z_1Z_2|>x)= x^{-\alpha} L_1(x)$, where $L_1$ is slowly varying function different from $L$;
see Embrechts and Goldie \cite{embrechts:goldie:1980}.
Write
\begin{equation*}
S_n=Z_1+\cdots +Z_n\,,\quad n\ge 1,
\end{equation*} and consider a sequence $(a_n)$ such that $\P(|Z|>a_n)\sim n^{-1}$.

\subsection{Large deviation results}
The following theorem can be found in
Nagaev \cite{nagaev:1979} and Cline and Hsing
\cite{cline:hsing:1998} for $\alpha>2$ and $\alpha\le 2$,
respectively; see also  Denisov et al.~\cite{denisov:dieker:shneer:2008}.
\begin{theorem}\label{thm:nagaev} 
Under the assumptions on the iid sequence $(Z_t)$
given above the following relation holds
\begin{equation*}
\sup_{x\ge c_n}\left|\dfrac{\P(S_n>x )}{n\P(|Z|>x)} -p_+ \right|\to 0\,,
\end{equation*}
where $(c_n)$ is any sequence satisfying $c_n/a_n\to  \infty$ for
$\alpha\le 2$ and $c_n\ge \sqrt{(\alpha-2)n\log n}$ for $\alpha>2$.
\end{theorem}

\subsection{Karamata theory for sums}
\begin{proposition}\label{prop:karasumsbeta}
Let $(c_n)$ be the threshold sequence in Theorem~\ref{thm:nagaev} for a
given $\alpha>0$,  and let $(d_n)$ be such that
$d_n/c_n\to\infty$ for $\alpha>2$ and $d_n=c_n$ for $\alpha\le 2$. Assume $0<\gamma< \alpha$.
Then we have for a sequence $x_n\ge d_n$
\begin{equation}
\E[|x_n^{-1} S_n|^\gamma \1_{\{|S_n|>x_n\}}] \sim \frac{\alpha}{\alpha -\gamma} n \P(|Z|>x_n), \qquad \nto.
\end{equation}
\end{proposition}
\begin{proof}
We use the notation $Y_n:=|x_n^{-1} S_n|$. Since $Y_n^\gamma \1_{\{Y_n>1\}}$ is a positive random variable one can write 
\begin{equation*}
\E[Y_n^\gamma \1_{\{Y_n>1\}}]=\int_0^\infty \P(Y_n^\gamma \1_{\{Y_n>1\}}>y) \dint y.
\end{equation*}
The probability inside the integral is 
\begin{equation*}
\begin{split}
\P(Y_n^\gamma\1_{\{Y_n>1\}}>y) &= \P(Y_n^\gamma \1_{\{Y_n>1\}}>y, Y_n>1)+\P(Y_n^\gamma \1_{\{Y_n>1\}}>y, Y_n<1)\\
&=\P(Y_n^\gamma >y, Y_n>1)= \P(Y_n > \max \{ y^{1/\gamma},1\})\\
&= \begin{cases}
\P(Y_n>1) & \text{if } y\le 1,\\ 
\P(Y_n>y^{1/\gamma}) & \text{if } y\ge 1.
\end{cases}
\end{split}
\end{equation*}
Therefore, using the uniform convergence result in Theorem~\ref{thm:nagaev}, we conclude that
\begin{equation*}
\begin{split}
\int_0^\infty \P(Y_n^\gamma \1_{\{Y_n>1\}}>y) \dint y &= \P(Y_n>1) + \int_1^\infty \P(Y_n>y^{1/\gamma}) \dint y\\
&\sim n \P(|Z|>x_n)+\int_1^\infty y^{-\frac{\alpha}{\gamma}} n \P(|Z|>x_n) \dint y\\
&= \frac{\alpha}{\alpha -\gamma} n \P(|Z|>x_n), \quad \nto.
\end{split}
\end{equation*}
\end{proof}

\subsection{A point process convergence result}
Assume that the conditions at the beginning of Appendix \ref{appendix:A} hold.
Consider a sequence of iid copies $(S_{n}^{(t)})_{t=1,2,\ldots}$
of $S_n$ and the sequence of point processes
\begin{equation*}
N_n= \sum_{t=1}^{p} \vep_{a_{np}^{-1} S_{n}^{(t)}}, \quad n=1,2,\ldots\,,
\end{equation*}
for an integer sequence $p=p_n\to\infty$. We assume that the state space of the
point processes $N_n$ is $\overline{\R}_0=[\R\cup\{\pm \infty\}]\backslash \{0\}$.
\begin{lemma}\label{lem:ppr} 
Assume $\alpha \in (0,2)$ and the
conditions of  Appendix~\ref{appendix:A}
on the iid sequence $(Z_t)$ and the normalizing sequence $(a_n)$. Then the limit relation
$N_n\cid N$ holds in the space of point measures on $\overline{\R}_0$
equipped with the vague topology (see
\cite{resnick:1987,resnick:2007})
for a  Poisson random measure $N$ with state space $\overline{\R}_0$ and intensity measure $\mu_\alpha(dx)=\alpha |x|^{-\alpha-1} (p_+ \1_{\{x>0\}}+ p_- \1_{\{x<0\}}) 
dx$.
\end{lemma}
\begin{proof}
According to Resnick \cite{resnick:1987}, Proposition 3.21, we need to
show that
$p\, \P(a_{np}^{-1}S_n\in \cdot)\civ \mu_\alpha
$,
where $\civ$ denotes vague convergence of Radon measures on  $\overline{\R}_0$.
Observe that we have $a_{np}/a_n\to\infty$ as $\nto$. This fact and
$\alpha\in (0,2)$ allow one to apply
Theorem~\ref{thm:nagaev}:
\begin{equation*}
\dfrac{\P( S_n >x a_{np})}{n\,\P(|Z|>  a_{np})}\to p_+ x^{-\alpha}
\quad \mbox{and}\quad \dfrac{\P( S_n \le -x a_{np})}{n\,\P(|Z|>
  a_{np})}\to p_-\, x^{-\alpha}\,,\quad x>0\,.
\end{equation*}
On the other hand, $n\,\P(|Z|>  a_{np})\sim p^{-1}$ as $\nto$.
This proves the lemma.
\end{proof}

\subsection{Auxiliary results}
%By construction, the entries of $\X_n$  are the order statistics $Z_{(1),np}^2, \ldots, Z_{(np),np}^2$. If we consider the first $k$ of them, we expect them to be larger than $a_{np}^{2\delta}$ with high probability, since $2\delta\le 2$ and we know that $a_{np}^{2}$ is the ``correct'' normalization according to the above point process results. This intuition will be confirmed by the next result.
Assume that the non-negative random variable $Z$ is regularly varying with index $\alpha \in (0,2)$  and $(a_n)$ is such that $n \,\P(Z>a_n)\sim 1$.
We also write 
\beao
Z_{(1)}\ge \cdots \ge Z_{(n)}\,,
\eeao
for the order statistics of the iid copies $Z_1,\ldots,Z_n$ of $Z$. 
\begin{lemma}\label{prop:seqk}
For every $\vep \in (0,0.5)$ there exists a sequence $k=k_n\to \infty$, $k<n$ such that 
\begin{equation*}
\lim_{\nto} \P(Z_{(k)} > a_{n}^{1-\vep}) = 1.
\end{equation*}
%{\blue ? possible choice for $(k_n)$ is $k_n= \vep \,\log (n)$.}
\end{lemma}
\begin{proof}[Proof of Lemma~\ref{prop:seqk}]
From the theory of order statistics we know that
\beao
\P(Z_{(k)} \le a_{n}^{1-\vep}) &=& \sum_{r=0}^{k-1}\binom{n}{r} \P(Z>a_{n}^{1-\vep})^r \; \P(Z\le a_{n}^{1-\vep})^{n-r}\\
&\le& \big(\P(Z\le a_{n}^{1-\vep})\big)^{n} \sum_{r=0}^{k-1} \frac{1}{r!} \Big( \frac{n \, \P(Z>a_{n}^{1-\vep})}{\P(Z\le a_{n}^{1-\vep})} \Big)^r.
\eeao
We observe that
\beao
\big(\P(Z\le a_{n}^{1-\vep})\big)^n \sim \ex^{- n\,\big[\P(Z>a_n^{1-\vep})-0.5(\P(Z>a_n^{1-\vep}))^2(1+o(1))\big]\,} 
\eeao
Writing $\Gamma(k)$ and $\Gamma(k,y)$ for the gamma and incompete gamma \fct s, we have
\beao
\ex^{-y}\sum_{r=0}^{k-1} \frac{y^r}{r!} = \frac{\Gamma(k,y)}{\Gamma(k)}=\P(\Gamma_k>y), \qquad y\ge 0\,,
\eeao
where $\Gamma_k=E_1+\cdots +E_k$, $k\ge 1$, for an iid standard exponential \seq\ $(E_i)$.
Therefore
\beao\lefteqn{\P(Z_{(k)} \le a_{n}^{1-\vep})}\\ &\le& c\,\ex^{- n\,\big[\P(Z>a_n^{1-\vep})
-0.5(\P(Z>a_n^{1-\vep}))^2(1+o(1))\big]+  \big[n\, \P(Z>a_{n}^{1-\vep})/\P(Z\le a_{n}^{1-\vep})\big]}\\
&&\P\big(\Gamma_k> n\, \P(Z>a_{n}^{1-\vep})/\P(Z\le a_{n}^{1-\vep})\big) \\&=&
c\,\ex^{O\big(n\,(\P(Z>a_n^{1-\vep}))^2\big)}\,\P\big(k^{-1}\Gamma_k> k^{-1}n\, \P(Z>a_{n}^{1-\vep})/\P(Z\le a_{n}^{1-\vep})\big) \,.
\eeao
The \rhs\ converges to zero if $2\vep<1$ and $k\le n^{\vep'}$ for some $\vep'<\vep$. 
\end{proof}
Now consider a  $p\times n$ random matrix $\bfZ$ with iid non-negative entries $Z_{it}$ and generic element $Z$ as specified above. The number of rows $p$ satisfies the growth condition \ref{eq:p}.
\par
We write for $\delta>0$,
\beam
B_{np}^\delta&=& \{ \mbox{There is a row of $\bfZ$ with at least two entries larger than $a_{np}^\delta$.} \}\,,\label{eq:bnp}
%\widetilde{B}_{np}^\delta&=& \{ \mbox{There is a column of $\bfZ$ with two entries larger than $a_{np}^\delta$.}\}\,.\nonumber
\eeam
\begin{lemma}\label{lem:onecolumn}
Assume that $p=p_n$ satisfies the growth condition \ref{eq:p} with $\beta\in [0,1]$. Then we have 
\begin{equation*}
\lim_{\nto} \P(B_{np}^\delta) =0 \quad \mbox{for all} \quad \delta>\frac{2+\beta}{2(1+\beta)}.
\end{equation*} 
\end{lemma}
\begin{proof}[Proof of Lemma~\ref{lem:onecolumn}]
Assume  $\delta>\frac{2+\beta}{2(1+\beta)}$
and consider the counting variables
\begin{equation*}
N_i=\sum_{t=1}^n \1_{\{Z_{it}>a_{np}^\delta\}}, \qquad i=1,\ldots,p.
\end{equation*}
Clearly, $N_i$ are iid $\text{Bin}(n, q)$   
with $q=q_n= \P(Z>a_{np}^\delta)\to 0$ as $\nto$  and
\beao
\P(B_{np}^\delta)& =& \P(\max_{i=1,\ldots,p} N_i \ge 2)\\ &=& 1-\big(\P(N_1\le 1)\big)^p\\
&=& 1- \big((1-q)^{n-1} (1+(n-1)q)\big)^p\,.
\eeao
Thus it remains to show that the \rhs\ converges to 0. Taking  logarithms, we get
\beao%\label{eq:lem5a11}
p\,\log \big((1-q)^{n-1} \,(1+(n-1)q)\big)= p\, [(n-1) \log(1-q)+\log(1+(n-1)q)].
\eeao
%We observe that $nq\to 0$ if $\delta>(1+\beta)^{-1}$. %This conditions holds if $\delta>(2+\beta)/(2(1+\beta))$.
A second order Taylor expansion of the logarithm yields 
\beam\label{eq:lem5a21}
p\,(n-1)\, \log(1-q)+p\,\log(1+(n-1)q)= p\,q+p\,\frac{(nq)^2}{2} +O(p\,(nq^2+(nq)^3)\,).
\eeam 
By the Potter bounds we conclude that
\eqref{eq:lem5a21} 
converges to zero if $\delta>\frac{2+\beta}{2(1+\beta)}$.
%\begin{equation}\label{eq:lem5a31}
%c\; p n^2q^2 \le c\; n^{\beta + \vep_1} n^2 n^{2(1+\beta-\vep_1)(-\delta+\vep_2)}= c\;n^{2+\beta+\vep_1+ 2(1+\beta-\vep_1)(-\delta+\vep_2)},
%\end{equation}
%where $\vep_1, \vep_2>0$ can be chosen arbitrarily small. The \rhs\ of \eqref{eq:lem5a31} converges to $0$ since the exponent of $n$ is negative, that is
%\begin{equation*}
%2+\beta+\vep_1+ 2(1+\beta-\vep_1)(-\delta+\vep_2)<0 \quad \Leftrightarrow \quad \delta>\frac{2+\beta+\vep_1}{2(1+\beta-\vep_1)}+\vep_2,
%\end{equation*}
%which is fulfilled for sufficiently small $\vep_1, \vep_2>0$. This finishes the proof of $\P(B_{np}^\delta) \to 0$.
The proof is complete. 
\end{proof}
For $\vep\in (0,1)$
define the events
\beao
A_i^{(n)}(\vep)&=& 
\Big\{ \sum_{t=1}^n Z_{it}- \max_{t=1, \ldots,n} Z_{it}> a_{np}^{1-\vep}  \Big\}, \qquad i=1,\ldots,p\,.
\eeao
The following result generalizes Lemma 5 in 
Auffinger et al.~\cite{auffinger:arous:peche:2009} (which in turn is a 
modified version of a result in Soshnikov~\cite{soshnikov:2004}) to the case of \regvary\ growth rates $(p_n)$.
The method of proof is different from the aforementioned literature.
\begin{lemma}\label{thm:keyresult}
Assume that $p=p_n=n^\beta \ell(n)$ where $\ell$ is a slowly varying function. 
Assume $ \beta\in (0,\infty)$ for $\alpha\in (0,1]$ and $\beta  \in (\alpha-1,\infty)$ for $\alpha\in [1,2)$.
There exists a constant $\vep\in (0,1)$ such that 
\begin{equation*}
\lim_{\nto} \P \Big( \bigcup_{i=1}^p A_i^{(n)}(\vep) \Big)=0\,.
\end{equation*}
\end{lemma}
\begin{proof}
Write $M_t=\max_{i=1,\ldots,t}Z_i$.
We observe that
\beao
\P \Big( \bigcup_{i=1}^p A_i^{(n)}(\vep) \Big)&\le & p\, \P(S_n-M_n>a_{np}^{1-\vep})\\
&=& n\,p\,\P(S_{n-1}>a_{np}^{1-\vep}\,,Z_n>M_{n-1})\\
&=&n\,p\,\int_0^\infty \P(S_{n-1}>a_{np}^{1-\vep}\,,z>M_{n-1})\, \dint \P(Z\le z)\,.
\eeao
We split the integration area into disjoint sets:
\beao
[0,\infty)= [0,a_n/h_n]\cup (a_n/h_n,a_{np}^{\gamma}]\cup (a_{np}^{\gamma},\infty) =\bigcup_{i=1}^3 B_i\,.
\eeao
We choose $h_n\to\infty$ \st\ $n\,\P(Z>a_n/h_n)\sim 2\log(np)$. Then 
\beam\label{eq:hp}
\log (np)- n\,\P(Z>a_n/h_n)\to -\infty\,,\qquad n\,\big(\P(Z>a_n/h_n)\big)^2\to 0\,.
\eeam
Moreover, choose $\gamma$ and $\vep>0$ fixed \st\ $\vep<1-(1\vee \alpha)/(1+\beta)$ and
\begin{itemize}
\item $\frac{1}{1+\beta}+\vep< \gamma <1-\frac{\vep}{1-\alpha}$ if $\alpha \in (0,1)$ and
\item $\frac{1}{1+\beta}+\vep< \gamma <1-\frac{2\vep}{2-\alpha}$ if $\alpha \in [1,2)$.
\end{itemize}
By virtue of \eqref{eq:hp} we have 
\beao
n\,p\,\int_{B_1} \P(S_{n-1}>a_{np}^{1-\vep}\,,z>M_{n-1})\,\dint \P(Z\le z)&\le& n\,p\,\P(M_{n-1}\le a_n/h_n)\\
&=& \ex^{\log (np)-n\,\P(Z>a_n/h_n)+o(1)}\to 0\,.
\eeao
By definition of $\vep$, we have $(a_n+n)/a_{np}^{1-\vep}\to 0$ for $\alpha\in (0,2)$. Therefore an application of
Theorem~\ref{thm:nagaev} yields
\beao
n\,p\,\int_{B_3} \P(S_{n-1}>a_{np}^{1-\vep}\,,z>M_{n-1})\, \dint \P(Z\le z)&\le &
n\,p\,\P(S_{n-1}>a_{np}^{1-\vep})\,\P(Z>a_{np}^{\gamma})\\
&\sim &\big(n\,p\,\P(Z>a_{np}^{1-\vep})\big)\,\big(n\,\P(Z>a_{np}^{\gamma})\big)\,.
\eeao
The \rhs\ converges to zero due to the property $\gamma>1/(1+\beta)+\vep$. 
\par
Now assume $\alpha\in (0,1)$. Then we have by Markov's inequality and Karamata's theorem,
\beao
\lefteqn{n\,p\,\int_{B_2} \P(S_{n-1}>a_{np}^{1-\vep}\,,z>M_{n-1})\, \dint \P(Z\le z)}\nonumber\\&\le&
\dfrac{n^2\,p}{a_{np}^{1-\vep}}\,\int_{B_2} \E [Z\1_{\{Z\le z\}}]  \dint \P(Z\le z)\\
&\le &\dfrac{n^2\,p}{a_{np}^{1-\vep}}\,\E [Z\1_{\{Z\le a_{np}^{\gamma}\}}] \,\P(Z>a_n/h_n)\\
&\sim & c\,\dfrac{n\,p}{a_{np}^{1-\vep}}\, \big[a_{np}^{\gamma}\,\P(Z>a_{np}^{\gamma})\big] \,\log(np)\,.
\eeao
An application of the Potter bounds and using the fact that $\gamma <1-\vep/(1-\alpha)$ shows that the \rhs\ converges to zero for the chosen $\vep$.
\par
Now assume $\alpha\in [1,2)$ and $\beta >\alpha-1$. Due to the latter condition we have $n/a_{np}^{1-\vep}\to 0$. 
We obtain by \v Cebyshev's inequality and Karamata's theorem,
\beao
\lefteqn{n\,p\,\int_{B_2} \P(S_{n-1}>a_{np}^{1-\vep}\,,z>M_{n-1})\, \dint \P(Z\le z)}\nonumber\\&\le&
n\,p\,\int_{B_2} \P\Big(\sum_{t=1}^n Z_t\,\1_{\{Z_t\le a_{np}^{\gamma}\}}-n\,\E \big[Z\,\1_{\{Z\le a_{np}^{\gamma}\}}\big]>a_{np}^{1-\vep}-n\,
\E \big[Z\,\1_{\{Z\le a_{np}^{\gamma}\}}\big])\, \dint \P(Z\le z)\\
&\le &n\,p\,\int_{B_2} 
\P\Big(\sum_{t=1}^n Z_t\,\1_{\{Z_t\le a_{np}^{\gamma}\}}-n\,\E \big[Z\,\1_{\{Z\le a_{np}^{\gamma}\}}\big]
>c\,a_{np}^{1-\vep}\Big)\, \dint \P(Z\le z)\\
&\le &n\,p \,\dfrac{\E [ Z^2\,\1_{\{Z\le a_{np}^{\gamma}\}}]}{a_{np}^{2(1-\vep)}}\,\big[n\,\P(Z>a_n/h_n)\big]\\
%n\,p\,\int_{B_2} \Big[\overline \Phi\Big(c\,a_{np}^{1-\vep}/\sqrt{p\,\E[Z^2\1_{\{Z\le a_p^{1+\delta}\}}]}\Big)\\
%&&+c\,\dfrac{1}{\sqrt{p}} \dfrac{\E[Z^3\1_{\{Z\le a_p^{1+\delta}\}}]}{(\E[Z^2\1_{\{Z\le a_p^{1+\delta}\}}])^{3/2}}\Big]
%\dfrac{1}{1+ c\,\Big(a_{np}^{1-\vep}/\sqrt{p \E[Z^2\1_{\{Z\le a_p^{1+\delta}\}}]}\Big)^3}
%\,d\P(Z\le z)\\
%&=&J_1+J_2\,,
&\sim &c\,n\,p \,\dfrac{a_{np}^{2\,\gamma} \P(Z>a_{np}^{\gamma})}{a_{np}^{2(1-\vep)}}\,\log(np)\,.
\eeao
The \rhs\ converges to zero since $\gamma <1-2\vep/(2-\alpha)$.
%We mentioned before that $p/a_{np}^{1-\vep}\to 0 $ for small $\vep$.
%If $\E [Z]<\infty$ the \rhs\ is bounded from below by  $c\,a_{np}^{1-\vep}$. Now we consider the remaining cases when $\E [Z]=\infty$.
%If $\alpha=1$, $\E \big[Z\,\1_{\{Z\le a_p^{1+\delta}\}}\big]$ is a \slvary\ \fct\ of $p$, hence the \rhs\ is again bounded from below by  
%$c\,a_{np}^{1-\vep}$. A calculation with Karamata's theorem shows that the same statement holds for $\alpha\in (0,1)$ and suitable choices
This finishes the proof.
\end{proof}

\subsection{Perturbation theory for eigenvectors}
We state Proposition~A.1 in Benaych-Georges and P\'{e}ch\'{e} \cite{benaych:peche}.
\begin{proposition}\label{prop:perturbation}
Let $\bfH$ be a Hermitean matrix and $\bfv$ a unit vector such that for some  $\la\in \R$, $\vep>0$,
\begin{equation*}
\bfH\,\bfv= \la\,\bfv + \vep\, \bfw\,,
\end{equation*}
where $\bfw$ is a unit vector such that $\bfw \perp \bfv$.
\begin{enumerate}
\item
Then $\bfH$ has an eigenvalue $\lambda_{\vep}$ \st\ $|\la-\la_\vep|\le \vep$.
\item
If $\bfH$ has only one eigenvalue $\la_\vep$ (counted with multiplicity) \st\ $|\la-\la_\vep|\le \vep$ and all other eigenvalues 
are at distance at least $d>\vep$ from $\la$. Then for a unit eigenvector $\bfv_\vep$ associated with $\lambda_{\vep}$ we have
\begin{equation*}
\ltwonorm{\bfv_{\vep} - \bfP_\bfv(\bfv_{\vep})} \le \frac{2\, \vep}{d-\vep}\,,
\end{equation*}
where $\bfP_{\bfv}$ denotes the orthogonal projection onto \rm{Span}$(\bfv)$.
\end{enumerate}
\end{proposition}

\section*{Acknowledgments}\setcounter{equation}{0}

We thank Richard A. Davis and Olivier Wintenberger for reading the manuscript and fruitful discussions. Special thanks go to Xiaolei Xie for providing the graphs.
%\newpage
%-----------------------------------------------------------------------
%Bibliography
%\bibliographystyle{amsplain}
%\bibliographystyle{unsrt}
%\bibliographystyle{alpha}
%\bibliographystyle{abbrv}
%\bibliographystyle{acm}
%\bibliographystyle{maths}
%\bibliographystyle{apalike}
%\bibliographystyle{amssiam}
%\bibliographystyle{unsrt}
%\bibliography{libraryjohannes}

\begin{thebibliography}{10}

\bibitem{anderson:1963}
{\sc Anderson, T.~W.}
\newblock Asymptotic theory for principal component analysis.
\newblock {\em Ann. Math. Statist. 34\/} (1963), 122--148.

\bibitem{auffinger:arous:peche:2009}
{\sc Auffinger, A., Ben~Arous, G., and P{\'e}ch{\'e}, S.}
\newblock Poisson convergence for the largest eigenvalues of heavy tailed
  random matrices.
\newblock {\em Ann. Inst. Henri Poincar\'e Probab. Stat. 45}, 3 (2009),
  589--610.

\bibitem{bai:silverstein:2010}
{\sc Bai, Z., and Silverstein, J.~W.}
\newblock {\em Spectral Analysis of Large Dimensional Random Matrices},
  second~ed.
\newblock Springer Series in Statistics. Springer, New York, 2010.

\bibitem{baisilv}
{\sc Bai, Z.~D., Silverstein, J.~W., and Yin, Y.~Q.}
\newblock A note on the largest eigenvalue of a large-dimensional sample
  covariance matrix.
\newblock {\em J. Multivariate Anal. 26}, 2 (1988), 166--168.

\bibitem{belinschi:dembo:guionnet:2009}
{\sc Belinschi, S., Dembo, A., and Guionnet, A.}
\newblock Spectral measure of heavy tailed band and covariance random matrices.
\newblock {\em Comm. Math. Phys. 289}, 3 (2009), 1023--1055.

\bibitem{arous:guionnet:2008}
{\sc Ben~Arous, G., and Guionnet, A.}
\newblock The spectrum of heavy tailed random matrices.
\newblock {\em Comm. Math. Phys. 278}, 3 (2008), 715--751.

\bibitem{benaych:peche}
{\sc Benaych-Georges, F., and P{\'e}ch{\'e}, S.}
\newblock Localization and delocalization for heavy tailed band matrices.
\newblock {\em Ann. Inst. Henri Poincar\'e Probab. Stat. 50}, 4 (2014),
  1385--1403.

\bibitem{bhatia:1997}
{\sc Bhatia, R.}
\newblock {\em Matrix Analysis}, vol.~169 of {\em Graduate Texts in
  Mathematics}.
\newblock Springer-Verlag, New York, 1997.

\bibitem{bingham:goldie:teugels:1987}
{\sc Bingham, N.~H., Goldie, C.~M., and Teugels, J.~L.}
\newblock {\em Regular Variation}, vol.~27 of {\em Encyclopedia of Mathematics
  and its Applications}.
\newblock Cambridge University Press, Cambridge, 1987.

\bibitem{chakrabarty:hazra:roy:20013}
{\sc Chakrabarty, A., Hazra, R.~S., and Roy, P.}
\newblock Maximum eigenvalue of symmetric random matrices with dependent heavy
  tailed entries.
\newblock {\em Available at {\tt http://arxiv.org/abs/1309.1407}\/} (2013).

\bibitem{cline:hsing:1998}
{\sc Cline, D. B.~H., and Hsing, T.}
\newblock Large deviation probabilities for sums of random variables with heavy
  or subexponential tails.
\newblock {\em Technical report. Statistics Dept., Texas A\&M University.\/}
  (1998).

\bibitem{davis:mikosch:heiny:xie:2015}
{\sc Davis, R.~A., Mikosch, T., Heiny, J., and Xie, X.}
\newblock Extreme value analysis for the sample autocovariance matrices of
  heavy-tailed multivariate time series.
\newblock {\em Extremes\/} (2016).

\bibitem{davis:mikosch:pfaffel:2015}
{\sc Davis, R.~A., Mikosch, T., and Pfaffel, O.}
\newblock Asymptotic theory for the sample covariance matrix of a heavy-tailed
  multivariate time series.
\newblock {\em Stochastic Process. Appl.\/} (2015).

\bibitem{davis:pfaffel:stelzer:2014}
{\sc Davis, R.~A., Pfaffel, O., and Stelzer, R.}
\newblock Limit theory for the largest eigenvalues of sample covariance
  matrices with heavy-tails.
\newblock {\em Stochastic Process. Appl. 124}, 1 (2014), 18--50.

\bibitem{denisov:dieker:shneer:2008}
{\sc Denisov, D., Dieker, A.~B., and Shneer, V.}
\newblock Large deviations for random walks under subexponentiality: the
  big-jump domain.
\newblock {\em Ann. Probab. 36}, 5 (2008), 1946--1991.

\bibitem{elkaroui:2003}
{\sc El~Karoui, N.}
\newblock On the largest eigenvalue of {W}ishart matrices with identity
  covariance when n,p and p/n tend to infinity.
\newblock {\em Available at {\tt http://arxiv.org/abs/math/0309355}\/} (2003).

\bibitem{embrechts:goldie:1980}
{\sc Embrechts, P., and Goldie, C.~M.}
\newblock On closure and factorization properties of subexponential and related
  distributions.
\newblock {\em J. Austral. Math. Soc. Ser. A 29}, 2 (1980), 243--256.

\bibitem{feller}
{\sc Feller, W.}
\newblock {\em An Introduction to Probability Theory and its Applications.
  {V}ol. {II}}.
\newblock John Wiley \& Sons, Inc., New York-London-Sydney, 1966.

\bibitem{geman}
{\sc Geman, S.}
\newblock A limit theorem for the norm of random matrices.
\newblock {\em Ann. Probab. 8}, 2 (1980), 252--261.

\bibitem{hult:2008}
{\sc Hult, H., and Samorodnitsky, G.}
\newblock Tail probabilities for infinite series of regularly varying random
  vectors.
\newblock {\em Bernoulli 14}, 3 (2008), 838--864.

\bibitem{janssen:mikosch:rezapour:xie:2016}
{\sc Janssen, A., Mikosch, T., Mohsen, R., and Xiaolei, X.}
\newblock The eigenvalues of the sample covariance matrix of a multivariate
  heavy-tailed stochastic volatility model.
\newblock {\em Available at {\tt https://arxiv.org/abs/1605.02563}\/} (2016).

\bibitem{johansson}
{\sc Johansson, K.}
\newblock Universality of the local spacing distribution in certain ensembles
  of {H}ermitian {W}igner matrices.
\newblock {\em Comm. Math. Phys. 215}, 3 (2001), 683--705.

\bibitem{johnstone:2001}
{\sc Johnstone, I.~M.}
\newblock On the distribution of the largest eigenvalue in principal components
  analysis.
\newblock {\em Ann. Statist. 29}, 2 (2001), 295--327.

\bibitem{lam:yao}
{\sc Lam, C., and Yao, Q.}
\newblock Factor modeling for high-dimensional time series: inference for the
  number of factors.
\newblock {\em Ann. Statist. 40}, 2 (2012), 694--726.

\bibitem{zongming:2012}
{\sc Ma, Z.}
\newblock Accuracy of the {T}racy-{W}idom limits for the extreme eigenvalues in
  white {W}ishart matrices.
\newblock {\em Bernoulli 18}, 1 (2012), 322--359.

\bibitem{muirhead}
{\sc Muirhead, R.~J.}
\newblock {\em Aspects of Multivariate Statistical Theory}.
\newblock John Wiley \& Sons, Inc., New York, 1982.
\newblock Wiley Series in Probability and Mathematical Statistics.

\bibitem{nagaev:1979}
{\sc Nagaev, S.~V.}
\newblock Large deviations of sums of independent random variables.
\newblock {\em Ann. Probab. 7}, 5 (1979), 745--789.

\bibitem{peche:2009}
{\sc P{\'e}ch{\'e}, S.}
\newblock Universality results for the largest eigenvalues of some sample
  covariance matrix ensembles.
\newblock {\em Probab. Theory Related Fields 143}, 3-4 (2009), 481--516.

\bibitem{resnick:2007}
{\sc Resnick, S.~I.}
\newblock {\em Heavy-Tail Phenomena: Probabilistic and Statistical Modeling}.
\newblock Springer Series in Operations Research and Financial Engineering.
  Springer, New York, 2007.

\bibitem{resnick:1987}
{\sc Resnick, S.~I.}
\newblock {\em Extreme Values, Regular Variation and Point Processes}.
\newblock Springer Series in Operations Research and Financial Engineering.
  Springer, New York, 2008.
\newblock Reprint of the 1987 original.

\bibitem{soshnikov:2004}
{\sc Soshnikov, A.}
\newblock Poisson statistics for the largest eigenvalues of {W}igner random
  matrices with heavy tails.
\newblock {\em Electron. Comm. Probab. 9\/} (2004), 82--91 (electronic).

\bibitem{soshnikov:2006}
{\sc Soshnikov, A.}
\newblock Poisson statistics for the largest eigenvalues in random matrix
  ensembles.
\newblock In {\em Mathematical physics of quantum mechanics}, vol.~690 of {\em
  Lecture Notes in Phys.} Springer, Berlin, 2006, pp.~351--364.

\bibitem{tao09b}
{\sc Tao, T., and Vu, V.}
\newblock Random matrices: universality of local eigenvalue statistics up to
  the edge.
\newblock {\em Comm. Math. Phys. 298}, 2 (2010), 549--572.

\bibitem{tao:vu:2012}
{\sc Tao, T., and Vu, V.}
\newblock Random covariance matrices: universality of local statistics of
  eigenvalues.
\newblock {\em Ann. Probab. 40}, 3 (2012), 1285--1315.

\bibitem{tracy:widom:2012}
{\sc Tracy, C.~A., and Widom, H.}
\newblock Distribution functions for largest eigenvalues and their
  applications.
\newblock In {\em Proceedings of the {I}nternational {C}ongress of
  {M}athematicians, {V}ol. {I} ({B}eijing, 2002)\/} (2002), Higher Ed. Press,
  Beijing, pp.~587--596.

\end{thebibliography}

\end{document}